\documentclass[12pt,reqno]{amsart}%
\usepackage{amsmath, amsfonts, amssymb, amsthm, amscd, amsbsy}
\usepackage{fancyhdr}
\usepackage[usenames,dvipsnames,svgnames,x11names,hyperref]{xcolor}
\usepackage{geometry}
\usepackage{graphicx}
\usepackage{hyperref}
\usepackage{amsmath}
\usepackage{amsfonts}
\usepackage{amssymb}%
\providecommand{\U}[1]{\protect\rule{.1in}{.1in}}
\hypersetup{backref=true,
pagebackref=true,
hyperindex=true,
colorlinks=true,
breaklinks=true,
urlcolor=Fuchsia,
linkcolor=NavyBlue,
bookmarks=true,
bookmarksopen=false,
filecolor=black,
citecolor=ForestGreen,
linkbordercolor=red
}
\newtheorem{theorem}{Theorem}[section]
\theoremstyle{plain}

\newtheorem{corollary}{Corollary}[section]

\newtheorem{lemma}{Lemma}[section]

\numberwithin{equation}{section}
\allowdisplaybreaks
\AtBeginDocument{\hypersetup{pdfborder={0 0 0.01}}}
\begin{document}
\title[CKN inequality and its stability]{Caffarelli-Kohn-Nirenberg identities, inequalities and their stabilities}
\author{Cristian Cazacu}
\address{Cristian Cazacu: $^{1}$Faculty of Mathematics and Computer Science \\
University of Bucha-rest\\
010014 Bucharest, Romania\\
\& $^{2}$Gheorghe Mihoc-Caius Iacob Institute of Mathematical\\
Statistics and Applied Mathematics of the Romanian Academy\\
050711 Bucharest, Romania }
\email{cristian.cazacu@fmi.unibuc.ro}
\author{Joshua Flynn}
\address{Joshua Flynn: Department of Mathematics\\
University of Connecticut\\
Storrs, CT 06269, USA}
\email{joshua.flynn@uconn.edu}
\author{Nguyen Lam}
\address{Nguyen Lam: School of Science and the Environment\\
Grenfell Campus, Memorial University of Newfoundland\\
Corner Brook, NL A2H5G4, Canada }
\email{nlam@grenfell.mun.ca}
\author{Guozhen Lu }
\address{Guozhen Lu: Department of Mathematics\\
University of Connecticut\\
Storrs, CT 06269, USA}
\email{guozhen.lu@uconn.edu}
\thanks{C. C. was partially supported by the Romanian Ministry of Research, Innovation
and Digitization, CNCS-UEFISCDI, project number PN-III-P1-1.1-TE-2021-1539,
within PNCDI III. N. L. was partially supported by an NSERC Discovery Grant.
G. L. were partially supported by a grant from the Simons Foundation.}
\date{\today}

\begin{abstract}
We set up a one-parameter family of inequalities that contains both the Hardy
inequalities (when the parameter is $1$) and the Caffarelli-Kohn-Nirenberg
inequalities (when the parameter is optimal). Moreover, we study these results
with the exact remainders to provide direct understandings to the sharp
constants, as well as the existence and non-existence of the optimizers of the
Hardy inequalities and Caffarelli-Kohn-Nirenberg inequalities. As an
application of our identities, we establish some sharp versions with optimal
constants and theirs attainability of the stability of the Heisenberg
Uncertainty Principle and several stability results of the
Caffarelli-Kohn-Nirenberg inequalities.

\end{abstract}
\maketitle

\section{Introduction}

Our starting point is the classical Hardy inequality that plays important
roles in many areas of analysis, mathematical physics and partial differential
equations: for $N\geq3,$ we have that
\begin{equation}
{\int\limits_{\mathbb{R}^{N}}}\left\vert \nabla u\right\vert ^{2}%
\mathrm{dx}\geq\left(  \frac{N-2}{2}\right)  ^{2}{\int\limits_{\mathbb{R}^{N}%
}}\frac{\left\vert u\right\vert ^{2}}{\left\vert x\right\vert ^{2}}%
\mathrm{dx},\text{ }u\in C_{0}^{\infty}\left(  \mathbb{R}^{N}\right)  ,
\label{Har}%
\end{equation}
with the sharp constant $\left(  \frac{N-2}{2}\right)  ^{2}$.

A very interesting fact about the Hardy inequalities that has attracted a lot
of attention is that though the constant $\left(  \frac{N-2}{2}\right)  ^{2}$
in (\ref{Har}) is optimal, the equality of (\ref{Har}) cannot occur for
nontrivial functions such that both sides of (\ref{Har}) are finite. For
instance, to explain for the aforementioned fact, many researchers have tried
to study the improvements of the Hardy type inequalities. In particular, in
the pioneering work \cite{BV97}, in order to study the stability of certain
singular solutions of nonlinear elliptic equations, Brezis and V\'{a}zquez
proved the following improved version of the Hardy inequality on bounded domains:

\medskip

\textbf{Theorem (Brezis-V\'{a}zquez \cite{BV97}).} \textit{For any bounded
domain }$\Omega\subset\mathbb{R}^{N}$\textit{, }$N\geq2$\textit{, and every
}$u\in H_{0}^{1}\left(  \Omega\right)  $\textit{, }%
\begin{equation}
{\int\limits_{\Omega}}\left\vert \nabla u\right\vert ^{2}\mathrm{dx}-\left(
\frac{N-2}{2}\right)  ^{2}{\int\limits_{\Omega}}\frac{\left\vert u\right\vert
^{2}}{\left\vert x\right\vert ^{2}}\mathrm{dx}\geq z_{0}^{2}\omega_{N}%
^{\frac{2}{N}}\left\vert \Omega\right\vert ^{-\frac{2}{N}}{\int\limits_{\Omega
}}\left\vert u\right\vert ^{2}\mathrm{dx} \label{1.3}%
\end{equation}
\textit{where }$z_{0}=2.4048...$\textit{ is the first zero of the Bessel
function }$J_{0}\left(  z\right)  $ \textit{and} $\omega_{N}$ \textit{is the
volume of the unit ball in} $\mathbb{R}^{N}$\textit{. The constant }$z_{0}%
^{2}\omega_{N}^{\frac{2}{N}}\left\vert \Omega\right\vert ^{-\frac{2}{N}}%
$\textit{ is optimal when }$\Omega$\textit{ is a ball but is not achieved in
the Sobolev space }$H_{0}^{1}\left(  \Omega\right)  $\textit{.}

\medskip

The fact that $z_{0}^{2}\omega_{N}^{\frac{2}{N}}\left\vert \Omega\right\vert
^{-\frac{2}{N}}$ is optimal when $\Omega$ is a ball, but still is not
achievable by nontrivial functions in (\ref{1.3}) led Brezis and V\'{a}zquez
to ask whether $z_{0}^{2}\omega_{N}^{\frac{2}{N}}\left\vert \Omega\right\vert
^{-\frac{2}{N}}$ is just the first term of an infinite series of remainder
terms. This question has drawn the attention and has been addressed by a lot
of researchers. The interested reader is referred to the monographs \cite{BEL,
GM1, KMP2007, KP, Maz11, OK}, for instance, that are standard references on
the subject. In particular, in \cite{FS08}, Frank and Seiringer provided a
general method in terms of nonlinear ground state representations to derive
the sharp local and nonlocal Hardy inequalities. We also note that the
improved Hardy type inequalities have also been investigated in the form of
identities in, for instance, \cite{DLL22, FLL21, LLZ19, LLZ20}.

In this paper, we will present another look at the Hardy type inequalities.
More precisely, we will set up a one-parameter family of inequalities in which
the Hardy inequalities correspond to the case when the parameter is $1$, while
if we optimize the parameter, we obtain the Caffarelli-Kohn-Nirenberg inequalities.

In other words, Hardy inequalities can be regarded as the non-optimal (scale
non-invariant) Caffarelli-Kohn-Nirenberg inequalities. It is worthy to note
that the Caffarelli-Kohn-Nirenberg inequalities have been established by
Caffarelli, Kohn and Nirenberg in their celebrated work \cite{CKN} to
generalize many well-known and important inequalities in analysis such as
Gagliardo-Nirenberg inequalities, Hardy-Sobolev inequalities, Nash's
inequalities, Sobolev inequalities, etc. Due to their important roles in many
areas of mathematics, the Caffarelli-Kohn-Nirenberg type inequalities and
their applications have seen a surge of research activity in recent years. We
do not attempt a survey of the extensive literature, but refer the reader to
\cite{CLZ21, DEFT15, DEL16, Dong18, Fly20, LL17, Lu97, Lu00, Wan22}, to name
just a few.

An important subfamily of the CKN inequality is the following $L^{2}%
$-Caffarelli-Kohn-Nirenberg inequality:%

\begin{equation}
{\int\limits_{\mathbb{R}^{N}}}\frac{|\nabla u|^{2}}{|x|^{2b}}\mathrm{dx}%
{\int\limits_{\mathbb{R}^{N}}}\frac{|u|^{2}}{|x|^{2a}}\mathrm{dx}\geq
C^{2}(N,a,b)\left(  {\int\limits_{\mathbb{R}^{N}}}\frac{|u|^{2}}{|x|^{a+b+1}%
}\mathrm{dx}\right)  ^{2}\text{, }u\in C_{0}^{\infty}(\mathbb{R}^{N}%
\setminus\{0\}). \label{CKN}%
\end{equation}
In particular, this subclass contains the Heisenberg Uncertainty Principle
($a=-1$, $b=0$), the Hydrogen Uncertainty Principle ($a=b=0$) and the Hardy
inequality ($a=1$, $b=0$), that play important roles in quantum mechanics. We
mention here that the optimal constant $C(N,a,b)>0$ of the $L^{2}%
$-Caffarelli-Kohn-Nirenberg inequality\ (\ref{CKN}) was first studied by Costa
in \cite{Cos08} for a particular range of parameters using
expanding-the-square method, and then by Catrina and Costa in \cite{CC09} for
the full range of parameters using spherical harmonics decomposition and a
Kelvin type transform. Very recently, the authors provided in \cite{CFL21} a
very simple and direct proof to derive the best constant $C(N,a,b)$ for the
whole range of parameters and to characterize all optimizers.

As mentioned earlier, we will prove in this article that the Hardy
inequalities and the Caffarelli-Kohn-Nirenberg inequalities belong to the same
family of inequalities in which the Caffarelli-Kohn-Nirenberg inequalities
appear to be the optimal ones. Hence, our results can be used to explain for
the attainability/unattainability of the sharp constants and the existence of
optimizers/virtual optimizers of the Hardy inequalities. Actually, we set up
these theorems with the exact remainders. Hence, our results can be applied to
identify and study the existence and non-existence of the optimizers of the
Hardy inequalities and the Caffarelli-Kohn-Nirenberg inequalities. We will
also show in this paper that our results can also be used to derive sharp
stability estimates of the Heisenberg Uncertainty Principle as well as some
stability versions of the Caffarelli-Kohn-Nirenberg inequalities.

\medskip

In 1985, Brezis and Lieb asked in \cite{BL85} whether the difference of the
two terms in the Sobolev inequalities controls the distance to the family of
extremal functions. This question has initiated the studies of quantitative
stability results for classical inequalities in mathematics that have been
investigated extensively and intensively in the literature. In \cite{BE91},
Bianchi and Egnell provided an affirmative answer to the question of Brezis
and Lieb for functions in $W^{1,2}\left(  \mathbb{R}^{N}\right)  $ by making
use of the fact that this function space is a Hilbert space: there is a
constant $c_{BE}>0$ such that
\[
{\int\limits_{\mathbb{R}^{N}}}|\nabla u|^{2}\mathrm{dx}-S_{N}\left(
{\int\limits_{\mathbb{R}^{N}}}|u|^{\frac{2N}{N-2}}\mathrm{dx}\right)
^{\frac{N-2}{N}}\geq c_{BE}\inf_{u^{\ast}}{\int\limits_{\mathbb{R}^{N}}%
}|\nabla u-\nabla u^{\ast}|^{2}\mathrm{dx}\text{.}%
\]
Here $S_{N}=\frac{1}{4}N\left(  N-2\right)  \left\vert \mathbb{S}%
^{N}\right\vert ^{\frac{2}{N}}$ is the optimal Sobolev constant and $u^{\ast
}\left(  x\right)  =\alpha\left(  \beta+\left\vert x-x_{0}\right\vert
^{2}\right)  ^{-\frac{N-2}{2}}$, $\alpha\in\mathbb{C}$, $\beta>0$, $x_{0}%
\in\mathbb{R}^{N}$, are the Aubin-Talenti functions. It is worth pointing out
that the stability constant $c_{BE}$ has not been investigated until very
recently. Indeed, in the paper \cite{BEFFL}, Dolbeault, Esteban, Figalli,
Frank and Loss established some lower and upper bounds for the stability
constant $c_{BE}$. In particular, in \cite{Kon22}, K\"{o}nig has proved that
$c_{BE}$ is strictly smaller the spectral gap constant $\frac{4}{N+4}$, which
is the best constant of the local stability of the Sobolev inequality
\cite{CFW13}.

The strategy in \cite{BE91} and its generalizations were also used by the
fourth author and Wei in \cite{LW00} to study the stability of the second
order Sobolev inequality, by Bartsch, Weth and Willem in \cite{BWW03} to
investigate the stability of the higher order Sobolev inequality, by Chen,
Frank and Weth in \cite{CFW13} to establish the stability of Sobolev
inequality for fractional orders, etc. The case on the Sobolev space
$W^{1,p}\left(  \mathbb{R}^{N}\right)  $, $p\neq2$, is much more complicated
and has just been established recently by, for instance, Cianchi, Fusco, Maggi
and Pratelli in \cite{CFMP09}, Figalli and Neumayer in \cite{FN19}, Neumayer
in \cite{Neu20} using new approaches.

In \cite{McV21}, McCurdy and Venkatraman studied the stability of the (scale
invariant) Heisenberg Uncertainty Principle. More precisely, they applied the
concentration-compactness arguments to show that there exist universal
constants $C_{1}>0$ and $C_{2}\left(  N\right)  >0$ such that for all
$u\in\left\{  u\in W^{1,2}\left(  \mathbb{R}^{N}\right)  :\left\Vert
xu\right\Vert _{2}<\infty\right\}  $:
\[
\delta_{2}\left(  u\right)  \geq C_{1}\left(  {\int\limits_{\mathbb{R}^{N}}%
}\left\vert u\right\vert ^{2}\mathrm{dx}\right)  d_{1}\left(  u,E\right)
^{2}+C_{2}\left(  N\right)  d_{1}\left(  u,E\right)  ^{4}\text{.}%
\]
Here
\[
\delta_{2}\left(  u\right)  :=\left(  {\int\limits_{\mathbb{R}^{N}}}\left\vert
\nabla u\right\vert ^{2}\mathrm{dx}\right)  \left(  {\int\limits_{\mathbb{R}%
^{N}}}\left\vert x\right\vert ^{2}\left\vert u\right\vert ^{2}\mathrm{dx}%
\right)  -\frac{N^{2}}{4}\left(  {\int\limits_{\mathbb{R}^{N}}}\left\vert
u\right\vert ^{2}\mathrm{dx}\right)  ^{2}%
\]
is the Heisenberg deficit, $E=\left\{  ce^{-\alpha\left\vert x\right\vert
^{2}}:c\in%
\mathbb{R}
\text{, }\alpha>0\right\}  $ is the set of the Gaussian functions, and
$d_{1}\left(  u,E\right)  :=\inf\left\{  \left(  {\int\limits_{\mathbb{R}^{N}%
}}\left\vert u-ce^{-\alpha\left\vert x\right\vert ^{2}}\right\vert
^{2}\mathrm{dx}\right)  ^{\frac{1}{2}};~c\in%
\mathbb{R}
\text{, }\alpha>0\right\}  $ is the $L^{2}$ distance to the set of the
optimizers. A short and constructive proof of this result has been given by
Fathi in \cite{Fat21} using the direct estimates via classical Gaussian
functional inequalities. More exactly, Fathi set up the following stability
version that provides explicit constants of the result in \cite{McV21}: for
all $u\in\left\{  u\in W^{1,2}\left(  \mathbb{R}^{N}\right)  :\left\Vert
xu\right\Vert _{2}<\infty\right\}  $, there holds%
\[
\delta_{2}\left(  u\right)  \geq\frac{1}{4}\left(  {\int\limits_{\mathbb{R}%
^{N}}}\left\vert u\right\vert ^{2}\mathrm{dx}\right)  d_{1}\left(  u,E\right)
^{2}+\frac{1}{16}d_{1}\left(  u,E\right)  ^{4}.
\]
However, this inequality is not optimal. In this paper, we will show that by
working on the scale non-invariant Heisenberg Uncertainty Principle and then
shifting to its optimal version using the one-parameter family of
inequalities, we are able to obtain a sharp version with optimal stability
constants of the above estimate (Theorem \ref{E3}) as a consequence of our
main results. Moreover, we also show that our optimal stability version can be
achieved by nontrivial functions.

It is also worth mentioning that the stability of the Gagliardo-Nirenberg
inequalities has also been studied in \cite{BDNS20, BDNS20a, CF13, DT16,
Ngu19, Seu16}, to name just a few. We also refer the readers to \cite{CF17,
CLT22, Dong181, DN21, FIL16, FJ15, FJ17, FMP13, FMP08, IM14, WW22}, and
references therein, for the stability results of many other functional and
geometric inequalities.

The second purpose of our paper is to study the stability of the
Caffarelli-Kohn-Nirenberg inequalities. We will use the aforementioned family
of inequalities and the following approach: We will first establish a weighted
version of the Poincar\'{e} inequality for the log-concave probability
measure. Then, by combining this new Poincar\'{e} inequality with the exact
remainders of the non-optimal (scale non-invariant) Caffarelli-Kohn-Nirenberg
inequalities, we obtain some versions of the stability of the non-optimal
(scale non-invariant) Caffarelli-Kohn-Nirenberg inequalities. Then, by
switching to the scale invariant ones, we deduce, among others, the following
version of the stability of the optimal Caffarelli-Kohn-Nirenberg inequalities:

\begin{theorem}
\label{E1}Let $0\leq b<\frac{N-2}{2}$, $a\leq\frac{Nb}{N-2}$ and
$a+b+1=\frac{2bN}{N-2}$. There exists a universal constant $C\left(
N,a,b\right)  >0$ such that
\begin{align*}
&  \left(  {\int\limits_{\mathbb{R}^{N}}}\frac{|u|^{2}}{|x|^{2a}}%
\mathrm{dx}\right)  ^{\frac{1}{2}}\left(  {\int\limits_{\mathbb{R}^{N}}}%
\frac{|\nabla u|^{2}}{|x|^{2b}}\mathrm{dx}\right)  ^{\frac{1}{2}}-\left\vert
\frac{N-a-b-1}{2}\right\vert \left(  {\int\limits_{\mathbb{R}^{N}}}%
\frac{|u|^{2}}{|x|^{a+b+1}}\mathrm{dx}\right) \\
&  \geq C\left(  N,a,b\right)  \inf_{\text{ }}\left\{  {\int
\limits_{\mathbb{R}^{N}}}\frac{\left\vert u-\alpha\exp(-\frac{\beta}%
{b+1-a}|x|^{b+1-a})\right\vert ^{2}}{\left\vert x\right\vert ^{a+b+1}%
}\mathrm{dx}\right\}  .
\end{align*}
Here the infimum is taken over the set of all $\alpha\in\mathbb{R}$
and$~\beta>0$ such that ${\int\limits_{\mathbb{R}^{N}}}\frac{|u|^{2}%
}{|x|^{a+b+1}}\mathrm{dx}={\int\limits_{\mathbb{R}^{N}}}\frac{|\alpha
\exp(-\frac{\beta}{b+1-a}|x|^{b+1-a})|^{2}}{|x|^{a+b+1}}\mathrm{dx}$.
\end{theorem}

We note that
\[
\delta_{1,a,b}\left(  u\right)  :=\left(  {\int\limits_{\mathbb{R}^{N}}}%
\frac{|u|^{2}}{|x|^{2a}}\mathrm{dx}\right)  ^{\frac{1}{2}}\left(
{\int\limits_{\mathbb{R}^{N}}}\frac{|\nabla u|^{2}}{|x|^{2b}}\mathrm{dx}%
\right)  ^{\frac{1}{2}}-\left\vert \frac{N-a-b-1}{2}\right\vert \left(
{\int\limits_{\mathbb{R}^{N}}}\frac{|u|^{2}}{|x|^{a+b+1}}\mathrm{dx}\right)
.
\]
is the Caffarelli-Kohn-Nirenberg deficit. Also, we will show in Section 3 that
$d_{1,a,b}\left(  u,E_{a,b}\right)  :=\inf_{v\in E_{a,b}}\left(
{\int\limits_{\mathbb{R}^{N}}}\frac{\left\vert u-v\right\vert ^{2}}{\left\vert
x\right\vert ^{a+b+1}}\mathrm{dx}\right)  ^{\frac{1}{2}}$ is the distance from
$u$ to $E_{a,b}$. Here
\[
E_{a,b}=\left\{  v(x)=\alpha\exp(-\frac{\beta}{b+1-a}|x|^{b+1-a})\text{;
}\alpha\in\mathbb{R},~\beta>0\right\}
\]
is the set of optimizers for the scale invariant Caffarelli-Kohn-Nirenberg
inequalities (\ref{CKN2}). Therefore, Theorem \ref{E1} implies that%
\[
\delta_{1,a,b}\left(  u\right)  \geq C\left(  N,a,b\right)  d_{1,a,b}\left(
u,E_{a,b}\right)  ^{2}.
\]

In the special case $a=-1,$ $b=0$ (that is, the Heisenberg Uncertainty
Principle), we actually obtain the explicit constants that do not depend on
the dimension:

\begin{theorem}
\label{E2}For all $u\in\left\{  u\in W^{1,2}\left(  \mathbb{R}^{N}\right)
:\left\Vert xu\right\Vert _{2}<\infty\right\}  $, we have%
\[
\delta_{1}\left(  u\right)  \geq\inf_{c\in\mathbb{R}\text{, }\alpha>0}\left\{
\left\Vert u-ce^{-\alpha\left\vert x\right\vert ^{2}}\right\Vert _{2}%
^{2}\right\}
\]
and \ \ \ \
\[
\delta_{1}\left(  u\right)  \geq\frac{1}{2}\inf_{c\in\mathbb{R}\text{, }%
\alpha>0}\left\{  \left\Vert u-ce^{-\alpha\left\vert x\right\vert ^{2}%
}\right\Vert _{2}^{2}:\left\Vert u\right\Vert _{2}=\left\Vert ce^{-\alpha
\left\vert x\right\vert ^{2}}\right\Vert _{2}^{2}\right\}  .
\]
These inequalities are sharp and the equalities can be attained by nontrivial functions.
\end{theorem}

Here we use the following Heisenberg deficit
\[
\delta_{1}\left(  u\right)  :=\left(  {\int\limits_{\mathbb{R}^{N}}}\left\vert
\nabla u\right\vert ^{2}\mathrm{dx}\right)  ^{\frac{1}{2}}\left(
{\int\limits_{\mathbb{R}^{N}}}\left\vert x\right\vert ^{2}\left\vert
u\right\vert ^{2}\mathrm{dx}\right)  ^{\frac{1}{2}}-\frac{N}{2}{\int
\limits_{\mathbb{R}^{N}}}\left\vert u\right\vert ^{2}\mathrm{dx}\text{.}%
\]

We note that the stability results in \cite{Fat21, McV21} use the Heisenberg
deficit
\[
\delta_{2}\left(  u\right)  =\left(  {\int\limits_{\mathbb{R}^{N}}}\left\vert
\nabla u\right\vert ^{2}\mathrm{dx}\right)  \left(  {\int\limits_{\mathbb{R}%
^{N}}}\left\vert x\right\vert ^{2}\left\vert u\right\vert ^{2}\mathrm{dx}%
\right)  -\frac{N^{2}}{4}\left(  {\int\limits_{\mathbb{R}^{N}}}\left\vert
u\right\vert ^{2}\mathrm{dx}\right)  ^{2}.
\]
Actually, our stability version with the Heisenberg deficit $\delta_{1}\left(
u\right)  $ implies the stability results with the Heisenberg deficit
$\delta_{2}\left(  u\right)  $ in \cite{Fat21, McV21}. Moreover, we are able
to establish the optimal stability constants and their attainabilities.
Indeed, from
\[
\delta_{1}\left(  u\right)  \geq d_{1}\left(  u,E\right)  ^{2},
\]
we deduce%
\[
\left(  {\int\limits_{\mathbb{R}^{N}}}\left\vert \nabla u\right\vert
^{2}\mathrm{dx}\right)  ^{\frac{1}{2}}\left(  {\int\limits_{\mathbb{R}^{N}}%
}\left\vert x\right\vert ^{2}\left\vert u\right\vert ^{2}\mathrm{dx}\right)
^{\frac{1}{2}}\geq\frac{N}{2}{\int\limits_{\mathbb{R}^{N}}}\left\vert
u\right\vert ^{2}\mathrm{dx}+d_{1}\left(  u,E\right)  ^{2}.
\]
Therefore%
\[
\left(  {\int\limits_{\mathbb{R}^{N}}}\left\vert \nabla u\right\vert
^{2}\mathrm{dx}\right)  \left(  {\int\limits_{\mathbb{R}^{N}}}\left\vert
x\right\vert ^{2}\left\vert u\right\vert ^{2}\mathrm{dx}\right)  \geq
\frac{N^{2}}{4}\left(  {\int\limits_{\mathbb{R}^{N}}}\left\vert u\right\vert
^{2}\mathrm{dx}\right)  ^{2}+N\left(  {\int\limits_{\mathbb{R}^{N}}}\left\vert
u\right\vert ^{2}\mathrm{dx}\right)  d_{1}\left(  u,E\right)  ^{2}%
+d_{1}\left(  u,E\right)  ^{4}.
\]
That is
\[
\delta_{2}\left(  u\right)  \geq N\left(  {\int\limits_{\mathbb{R}^{N}}%
}\left\vert u\right\vert ^{2}\mathrm{dx}\right)  d_{1}\left(  u,E\right)
^{2}+d_{1}\left(  u,E\right)  ^{4}.
\]
Similarly, let $d_{2}\left(  u,E\right)  :=\inf_{c\in\mathbb{R}\text{, }%
\alpha>0}\left\{  \left\Vert u-ce^{-\alpha\left\vert x\right\vert ^{2}%
}\right\Vert _{2}:\left\Vert u\right\Vert _{2}=\left\Vert ce^{-\alpha
\left\vert x\right\vert ^{2}}\right\Vert _{2}^{2}\right\}  \geq d_{1}\left(
u,E\right)  $. Then since
\[
\delta_{1}\left(  u\right)  \geq\frac{1}{2}d_{2}\left(  u,E\right)  ^{2},
\]
we deduce%
\[
\left(  {\int\limits_{\mathbb{R}^{N}}}\left\vert \nabla u\right\vert
^{2}\mathrm{dx}\right)  ^{\frac{1}{2}}\left(  {\int\limits_{\mathbb{R}^{N}}%
}\left\vert x\right\vert ^{2}\left\vert u\right\vert ^{2}\mathrm{dx}\right)
^{\frac{1}{2}}\geq\frac{N}{2}{\int\limits_{\mathbb{R}^{N}}}\left\vert
u\right\vert ^{2}\mathrm{dx}+\frac{1}{2}d_{2}\left(  u,E\right)  ^{2}.
\]
Therefore%
\[
\left(  {\int\limits_{\mathbb{R}^{N}}}\left\vert \nabla u\right\vert
^{2}\mathrm{dx}\right)  \left(  {\int\limits_{\mathbb{R}^{N}}}\left\vert
x\right\vert ^{2}\left\vert u\right\vert ^{2}\mathrm{dx}\right)  \geq
\frac{N^{2}}{4}\left(  {\int\limits_{\mathbb{R}^{N}}}\left\vert u\right\vert
^{2}\mathrm{dx}\right)  ^{2}+\frac{N}{2}\left(  {\int\limits_{\mathbb{R}^{N}}%
}\left\vert u\right\vert ^{2}\mathrm{dx}\right)  d_{2}\left(  u,E\right)
^{2}+\frac{1}{4}d_{2}\left(  u,E\right)  ^{4}.
\]
That is
\[
\delta_{2}\left(  u\right)  \geq\frac{N}{2}\left(  {\int\limits_{\mathbb{R}%
^{N}}}\left\vert u\right\vert ^{2}\mathrm{dx}\right)  d_{2}\left(  u,E\right)
^{2}+\frac{1}{4}d_{2}\left(  u,E\right)  ^{4}.
\]
Therefore, we have the following sharp results:

\begin{theorem}
\label{E3}For all $u\in\left\{  u\in W^{1,2}\left(  \mathbb{R}^{N}\right)
:\left\Vert xu\right\Vert _{2}<\infty\right\}  $, we have%
\[
\delta_{2}\left(  u\right)  \geq N\left(  {\int\limits_{\mathbb{R}^{N}}%
}\left\vert u\right\vert ^{2}\mathrm{dx}\right)  d_{1}\left(  u,E\right)
^{2}+d_{1}\left(  u,E\right)  ^{4}%
\]
and
\[
\delta_{2}\left(  u\right)  \geq\frac{N}{2}\left(  {\int\limits_{\mathbb{R}%
^{N}}}\left\vert u\right\vert ^{2}\mathrm{dx}\right)  d_{2}\left(  u,E\right)
^{2}+\frac{1}{4}d_{2}\left(  u,E\right)  ^{4}.
\]
These inequalities are sharp and the equalities can be attained by nontrivial functions.
\end{theorem}

It was also showed in \cite{McV21} that for any two nonnegative constants
$C_{1}$ and $C_{2}$ such that $C_{1}^{2}+C_{2}^{2}>0$, there exists
$u\in\left\{  u\in W^{1,2}\left(  \mathbb{R}^{N}\right)  :\left\Vert
xu\right\Vert _{2}<\infty\right\}  $, $\left\Vert u\right\Vert _{2}=1$ and
$u^{\ast}\in E$ such that%
\[
\delta_{2}\left(  u\right)  \leq C_{1}\left\Vert \nabla\left(  u-u^{\ast
}\right)  \right\Vert _{2}^{2}+C_{2}\left\Vert x\left(  u-u^{\ast}\right)
\right\Vert _{2}^{2}.
\]
That is, there is no quantitative stability version of the scale invariant
Heisenberg Uncertainty Principle if we use the norm $\left\Vert \nabla\left(
u-u^{\ast}\right)  \right\Vert _{2}$ or $\left\Vert x\left(  u-u^{\ast
}\right)  \right\Vert _{2}$ as distance functions. In this paper, we will also
show that this is not the case for the scale non-invariant Heisenberg
Uncertainty Principle:

\begin{theorem}
\label{E4}For all $u\in\left\{  u\in W^{1,2}\left(  \mathbb{R}^{N}\right)
:\left\Vert xu\right\Vert _{2}<\infty\right\}  $, then%
\begin{align*}
&  {\int\limits_{\mathbb{R}^{N}}}\left\vert \nabla u\right\vert ^{2}%
\mathrm{dx}+{\int\limits_{\mathbb{R}^{N}}}\left\vert x\right\vert
^{2}\left\vert u\right\vert ^{2}\mathrm{dx}-N{\int\limits_{\mathbb{R}^{N}}%
}\left\vert u\right\vert ^{2}\mathrm{dx}\\
&  \geq\frac{2}{N+3}\inf_{c\in\mathbb{R}}\left(  {\int\limits_{\mathbb{R}^{N}%
}}\left\vert \nabla\left(  u-ce^{-\frac{1}{2}\left\vert x\right\vert ^{2}%
}\right)  \right\vert ^{2}\mathrm{dx}+{\int\limits_{\mathbb{R}^{N}}}\left\vert
x\right\vert ^{2}\left\vert u-ce^{-\frac{1}{2}\left\vert x\right\vert ^{2}%
}\right\vert ^{2}\mathrm{dx}+{\int\limits_{\mathbb{R}^{N}}}\left\vert
u-ce^{-\frac{1}{2}\left\vert x\right\vert ^{2}}\right\vert ^{2}\mathrm{dx}%
\right)  .
\end{align*}
Moreover, the inequality is sharp and the equality can be attained by
nontrivial functions.
\end{theorem}

We end this introduction with the following remark. Recently, Anh Do and the
second, third and fourth authors have established in \cite{DFLL22} the
stability of the $L^{p}$-Caffarelli-Kohn-Nirenberg inequalities by proving the
$L^{p}$-Caffarelli-Kohn-Nirenberg identities and the weighted version of the
$L^{p}$-Poincar\'{e} inequality for log-concave measures.

Our paper is organized as follows: In Section 2, we will set up a general
identity and use it to establish the $L^{2}$-Hardy and $L^{2}$%
-Caffarelli-Kohn-Nirenberg identities. Several examples will also be provided
in Section 2. In Section 3, we will use some of these identities to study the
sharp stability results for the Heisenberg Uncertainty Principle as well as
several quantitative results about the stability of the
Caffarelli-Kohn-Nirenberg inequalities.

\section{$L^{2}$-Hardy and $L^{2}$-Caffarelli-Kohn-Nirenberg identities on
$\mathbb{R}^{N}$}

In this section we will set up general versions of the $L^{2}$-Hardy and
$L^{2}$-Caffarelli-Kohn-Nirenberg identities on $\mathbb{R}^{N}$. Denote
$\mathcal{R}u\left(  x\right)  :=$ $\frac{x}{\left\vert x\right\vert }%
\cdot\nabla u\left(  x\right)  $. This is the radial derivative. That is, in
the polar coordinate $x=r\sigma$, $\mathcal{R}u=\partial_{r}u$. Let $H\in
C^{1}\left(  0,R\right)  $, $0<R\leq\infty$. By direct computation, we have
\begin{align*}
\operatorname{div}\left(  H\left(  \left\vert x\right\vert \right)  \frac
{x}{\left\vert x\right\vert }\right)   &  =\sum_{i=1}^{N}\frac{\partial
}{\partial x_{i}}\left(  H\left(  \left\vert x\right\vert \right)  \frac
{x_{i}}{\left\vert x\right\vert }\right) \\
&  =N\frac{H\left(  \left\vert x\right\vert \right)  }{\left\vert x\right\vert
}+\frac{\left\vert x\right\vert H^{\prime}\left(  \left\vert x\right\vert
\right)  -H\left(  \left\vert x\right\vert \right)  }{\left\vert x\right\vert
}\\
&  =H^{\prime}\left(  \left\vert x\right\vert \right)  +\left(  N-1\right)
\frac{H\left(  \left\vert x\right\vert \right)  }{\left\vert x\right\vert }.
\end{align*}
Hence for $u\in C_{0}^{\infty}\left(  B_{R}\setminus\left\{  0\right\}
\right)  $, we have by the Divergence Theorem that%
\begin{align}
&  -{\int\limits_{B_{R}}}\left[  H^{\prime}\left(  \left\vert x\right\vert
\right)  +\left(  N-1\right)  \frac{H\left(  \left\vert x\right\vert \right)
}{\left\vert x\right\vert }\right]  \left\vert u\right\vert ^{2}%
\mathrm{dx}\nonumber\\
&  =-{\int\limits_{B_{R}}}\operatorname{div}\left(  H\left(  \left\vert
x\right\vert \right)  \frac{x}{\left\vert x\right\vert }\right)  \left\vert
u\right\vert ^{2}\mathrm{dx}\nonumber\\
&  ={\int\limits_{B_{R}}}H\left(  \left\vert x\right\vert \right)  \frac
{x}{\left\vert x\right\vert }\cdot\nabla\left\vert u\right\vert ^{2}%
\mathrm{dx}\nonumber\\
&  ={\int\limits_{B_{R}}}2H\left(  \left\vert x\right\vert \right)  u\left(
x\right)  \left(  \frac{x}{\left\vert x\right\vert }\cdot\nabla u\left(
x\right)  \right)  \mathrm{dx}\nonumber\\
&  ={\int\limits_{B_{R}}}2H\left(  \left\vert x\right\vert \right)  u\left(
x\right)  \mathcal{R}u\left(  x\right)  \mathrm{dx}. \label{2.1}%
\end{align}
We will now apply the above identity for
\[
H\left(  r\right)  =A\left(  r\right)  B\left(  r\right)  .
\]
Note that in this case,
\begin{align*}
&  H^{\prime}\left(  r\right)  +\left(  N-1\right)  \frac{H\left(  r\right)
}{r}\\
&  =A^{\prime}\left(  r\right)  B\left(  r\right)  +A\left(  r\right)
B^{\prime}\left(  r\right)  +\left(  N-1\right)  \frac{A\left(  r\right)
B\left(  r\right)  }{r}\\
&  =\frac{1}{\alpha^{2}}B^{2}\left(  r\right)  +\left[  C\left(  r\right)
+B^{2}\left(  r\right)  -\frac{1}{\alpha^{2}}B^{2}\left(  r\right)  \right]
\end{align*}
for some $\alpha\neq0$. Here
\[
C\left(  r\right)  =\left(  A\left(  r\right)  B\left(  r\right)  \right)
^{\prime}+\left(  N-1\right)  \frac{A\left(  r\right)  B\left(  r\right)  }%
{r}-B^{2}\left(  r\right)  \text{.}%
\]
Then (\ref{2.1}) gives
\begin{align}
&  -{\int\limits_{B_{R}}}\frac{1}{\alpha^{2}}B^{2}\left(  \left\vert
x\right\vert \right)  \left\vert u\left(  x\right)  \right\vert ^{2}%
\mathrm{dx}-{\int\limits_{B_{R}}}\left[  C\left(  \left\vert x\right\vert
\right)  +B^{2}\left(  \left\vert x\right\vert \right)  -\frac{1}{\alpha^{2}%
}B^{2}\left(  \left\vert x\right\vert \right)  \right]  \left\vert u\left(
x\right)  \right\vert ^{2}\mathrm{dx}\nonumber\\
&  ={\int\limits_{B_{R}}}2\left(  \frac{1}{\alpha}B\left(  \left\vert
x\right\vert \right)  u\left(  x\right)  \right)  \left(  \alpha A\left(
\left\vert x\right\vert \right)  \mathcal{R}u\left(  x\right)  \right)
\mathrm{dx}. \label{2.2}%
\end{align}
Equivalently, by using the identity $-2ab=a^{2}+b^{2}-\left(  a+b\right)
^{2}$:%
\begin{align*}
&  \alpha^{2}{\int\limits_{B_{R}}}A^{2}\left(  \left\vert x\right\vert
\right)  \left\vert \nabla u\left(  x\right)  \right\vert ^{2}\mathrm{dx}\\
&  ={\int\limits_{B_{R}}}\left[  C\left(  \left\vert x\right\vert \right)
+B^{2}\left(  \left\vert x\right\vert \right)  -\frac{1}{\alpha^{2}}%
B^{2}\left(  \left\vert x\right\vert \right)  \right]  \left\vert u\left(
x\right)  \right\vert ^{2}\mathrm{dx}+{\int\limits_{B_{R}}}\left\vert \alpha
A\left(  \left\vert x\right\vert \right)  \nabla u\left(  x\right)  +\frac
{1}{\alpha}B\left(  \left\vert x\right\vert \right)  u\left(  x\right)
\frac{x}{\left\vert x\right\vert }\right\vert ^{2}\mathrm{dx}%
\end{align*}
and
\begin{align*}
&  \alpha^{2}{\int\limits_{B_{R}}}A^{2}\left(  \left\vert x\right\vert
\right)  \left\vert \mathcal{R}u\left(  x\right)  \right\vert ^{2}%
\mathrm{dx}\\
&  ={\int\limits_{B_{R}}}\left[  C\left(  \left\vert x\right\vert \right)
+B^{2}\left(  \left\vert x\right\vert \right)  -\frac{1}{\alpha^{2}}%
B^{2}\left(  \left\vert x\right\vert \right)  \right]  \left\vert u\left(
x\right)  \right\vert ^{2}\mathrm{dx}+{\int\limits_{B_{R}}}\left\vert \alpha
A\left(  \left\vert x\right\vert \right)  \mathcal{R}u\left(  x\right)
+\frac{1}{\alpha}B\left(  \left\vert x\right\vert \right)  u\left(  x\right)
\right\vert ^{2}\mathrm{dx}.
\end{align*}
Hence, we get the following families of identities:

\begin{lemma}
\label{key}\textit{Let }$0<R\leq\infty$\textit{, }$A$\textit{ and }$B$\textit{
be }$C^{1}$-\textit{functions on }$\left(  0,R\right)  $ and let
\[
C\left(  r\right)  =\left(  A\left(  r\right)  B\left(  r\right)  \right)
^{\prime}+\left(  N-1\right)  \frac{A\left(  r\right)  B\left(  r\right)  }%
{r}-B^{2}\left(  r\right)  \text{.}%
\]
Then for all $\alpha\in\mathbb{R}\setminus\left\{  0\right\}  $ and $u\in
C_{0}^{\infty}\left(  B_{R}\setminus\left\{  0\right\}  \right)  $, we have%
\begin{align}
&  \alpha^{2}{\int\limits_{B_{R}}}A^{2}\left(  \left\vert x\right\vert
\right)  \left\vert \mathcal{R}u\left(  x\right)  \right\vert ^{2}%
\mathrm{dx}+\frac{1}{\alpha^{2}}{\int\limits_{B_{R}}}B^{2}\left(  \left\vert
x\right\vert \right)  \left\vert u\left(  x\right)  \right\vert ^{2}%
\mathrm{dx}\nonumber\\
&  ={\int\limits_{B_{R}}}\left[  C\left(  \left\vert x\right\vert \right)
+B^{2}\left(  \left\vert x\right\vert \right)  \right]  \left\vert
u\right\vert ^{2}\mathrm{dx}+{\int\limits_{B_{R}}}\left\vert \alpha A\left(
\left\vert x\right\vert \right)  \mathcal{R}u+\frac{1}{\alpha}B\left(
\left\vert x\right\vert \right)  u\right\vert ^{2}\mathrm{dx} \label{CKN2.1}%
\end{align}
and%
\begin{align}
&  \alpha^{2}{\int\limits_{B_{R}}}A^{2}\left(  \left\vert x\right\vert
\right)  \left\vert \nabla u\right\vert ^{2}\mathrm{dx}+\frac{1}{\alpha^{2}%
}{\int\limits_{B_{R}}}B^{2}\left(  \left\vert x\right\vert \right)  \left\vert
u\right\vert ^{2}\mathrm{dx}\nonumber\\
&  ={\int\limits_{B_{R}}}\left[  C\left(  \left\vert x\right\vert \right)
+B^{2}\left(  \left\vert x\right\vert \right)  \right]  \left\vert u\left(
x\right)  \right\vert ^{2}\mathrm{dx}+{\int\limits_{B_{R}}}\left\vert \alpha
A\left(  \left\vert x\right\vert \right)  \nabla u\left(  x\right)  +\frac
{1}{\alpha}B\left(  \left\vert x\right\vert \right)  u\left(  x\right)
\frac{x}{\left\vert x\right\vert }\right\vert ^{2}\mathrm{dx}. \label{CKN2.2}%
\end{align}

\end{lemma}

By choosing $\alpha=1$, we obtain the $L^{2}$-Hardy identities while by
optimizing $\alpha$, we get the $L^{2}$-Caffarelli-Kohn-Nirenberg identities.

\subsection{$L^{2}$-Hardy identities}

By choose $\alpha=1$ in the Lemma \ref{key}, we have the following $L^{2}%
$-Hardy identities

\begin{theorem}
\label{T1}\textit{Let }$0<R\leq\infty$\textit{, }$A$\textit{ and }$B$\textit{
be }$C^{1}$-\textit{functions on }$\left(  0,R\right)  $ and let
\[
C\left(  r\right)  =\left(  A\left(  r\right)  B\left(  r\right)  \right)
^{\prime}+\left(  N-1\right)  \frac{A\left(  r\right)  B\left(  r\right)  }%
{r}-B^{2}\left(  r\right)  \text{.}%
\]
Then for all $u\in C_{0}^{\infty}\left(  B_{R}\setminus\left\{  0\right\}
\right)  $, we have%
\begin{equation}
{\int\limits_{B_{R}}}A^{2}\left(  \left\vert x\right\vert \right)  \left\vert
\mathcal{R}u\right\vert ^{2}\mathrm{dx}={\int\limits_{B_{R}}}C\left(
\left\vert x\right\vert \right)  \left\vert u\right\vert ^{2}\mathrm{dx}%
+{\int\limits_{B_{R}}}\left\vert A\left(  \left\vert x\right\vert \right)
\mathcal{R}u+B\left(  \left\vert x\right\vert \right)  u\right\vert
^{2}\mathrm{dx} \label{H2.1}%
\end{equation}
and
\begin{equation}
{\int\limits_{B_{R}}}A^{2}\left(  \left\vert x\right\vert \right)  \left\vert
\nabla u\right\vert ^{2}\mathrm{dx}={\int\limits_{B_{R}}}C\left(  \left\vert
x\right\vert \right)  \left\vert u\right\vert ^{2}\mathrm{dx}+{\int
\limits_{B_{R}}}\left\vert A\left(  \left\vert x\right\vert \right)  \nabla
u+B\left(  \left\vert x\right\vert \right)  u\frac{x}{\left\vert x\right\vert
}\right\vert ^{2}\mathrm{dx}. \label{H2.2}%
\end{equation}

\end{theorem}

Here are some examples that follow from Theorem \ref{T1} immediately by
checking the appropriate pairs of $A, B$ and $C$.

\begin{corollary}
\label{c1}Let $A=1$, $B=\frac{N-2}{2}\frac{1}{r}$. Then $C=\left(  \frac
{N-2}{2}\frac{1}{r}\right)  ^{2}$. Then by the Hardy identities (\ref{H2.1})
and (\ref{H2.2}), for all $u\in C_{0}^{\infty}\left(  \mathbb{R}^{N}%
\setminus\left\{  0\right\}  \right)  :$%
\begin{align*}
{\int\limits_{\mathbb{R}^{N}}}\left\vert \mathcal{R}u\right\vert
^{2}\mathrm{dx}  &  =\left(  \frac{N-2}{2}\right)  ^{2}{\int
\limits_{\mathbb{R}^{N}}}\frac{\left\vert u\right\vert ^{2}}{\left\vert
x\right\vert ^{2}}\mathrm{dx}+{\int\limits_{\mathbb{R}^{N}}}\left\vert
\mathcal{R}u+\frac{N-2}{2}\frac{1}{\left\vert x\right\vert }u\right\vert
^{2}\mathrm{dx}\\
{\int\limits_{\mathbb{R}^{N}}}\left\vert \nabla u\right\vert ^{2}\mathrm{dx}
&  =\left(  \frac{N-2}{2}\right)  ^{2}{\int\limits_{\mathbb{R}^{N}}}%
\frac{\left\vert u\right\vert ^{2}}{\left\vert x\right\vert ^{2}}%
\mathrm{dx}+{\int\limits_{\mathbb{R}^{N}}}\left\vert \nabla u+\frac{N-2}%
{2}\frac{1}{\left\vert x\right\vert }u\frac{x}{\left\vert x\right\vert
}\right\vert ^{2}\mathrm{dx}.
\end{align*}

\end{corollary}

\begin{corollary}
Let \label{c2}$A=r^{-\frac{\lambda}{2}}$, $B=\frac{N-\lambda-2}{2}%
r^{-\frac{\lambda}{2}-1}$. Then $C=B^{2}=\left(  \frac{N-\lambda-2}%
{2}r^{-\frac{\lambda}{2}-1}\right)  ^{2}$. Hence using the Hardy identities
(\ref{H2.1}) and (\ref{H2.2}), we obtain for all $u\in C_{0}^{\infty}\left(
\mathbb{R}^{N}\setminus\left\{  0\right\}  \right)  :$
\begin{align*}
&  {\int\limits_{\mathbb{R}^{N}}}\frac{1}{\left\vert x\right\vert ^{\lambda}%
}\left\vert \mathcal{R}u\right\vert ^{2}\mathrm{dx}-\left(  \frac{N-\lambda
-2}{2}\right)  ^{2}{\int\limits_{\mathbb{R}^{N}}}\frac{1}{\left\vert
x\right\vert ^{\lambda+2}}\left\vert u\right\vert ^{2}\mathrm{dx}\\
&  ={\int\limits_{\mathbb{R}^{N}}}\left\vert \frac{N-\lambda-2}{2}\frac
{1}{\left\vert x\right\vert ^{\frac{\lambda}{2}+1}}u+\frac{1}{\left\vert
x\right\vert ^{\frac{\lambda}{2}}}\mathcal{R}u\right\vert ^{2}\mathrm{dx}\\
&  ={\int\limits_{\mathbb{R}^{N}}}\left\vert \frac{1}{\left\vert x\right\vert
^{\frac{N-2}{2}}}\mathcal{R}\left(  \left\vert x\right\vert ^{\frac
{N-\lambda-2}{2}}u\right)  \right\vert ^{2}\mathrm{dx}%
\end{align*}
and%
\begin{align*}
&  {\int\limits_{\mathbb{R}^{N}}}\frac{1}{\left\vert x\right\vert ^{\lambda}%
}\left\vert \nabla u\right\vert ^{2}\mathrm{dx}-\left(  \frac{N-\lambda-2}%
{2}\right)  ^{2}{\int\limits_{\mathbb{R}^{N}}}\frac{1}{\left\vert x\right\vert
^{\lambda+2}}\left\vert u\right\vert ^{2}\mathrm{dx}\\
&  ={\int\limits_{\mathbb{R}^{N}}}\left\vert \frac{N-\lambda-2}{2}\frac
{1}{\left\vert x\right\vert ^{\frac{\lambda}{2}+1}}u\frac{x}{\left\vert
x\right\vert }+\frac{1}{\left\vert x\right\vert ^{\frac{\lambda}{2}}}\nabla
u\right\vert ^{2}\mathrm{dx}\\
&  ={\int\limits_{\mathbb{R}^{N}}}\left\vert \frac{1}{\left\vert x\right\vert
^{\frac{N-2}{2}}}\nabla\left(  \left\vert x\right\vert ^{\frac{N-\lambda-2}%
{2}}u\right)  \right\vert ^{2}\mathrm{dx}.
\end{align*}

\end{corollary}

\begin{corollary}
\label{c3}$A=1$, $B=r$, and $C=N-r^{2}$. By the Hardy identities (\ref{H2.1})
and (\ref{H2.2}), for all $u\in C_{0}^{\infty}\left(  \mathbb{R}^{N}\right)
:$%
\begin{align*}
{\int\limits_{\mathbb{R}^{N}}}\left\vert \mathcal{R}u\right\vert
^{2}\mathrm{dx}  &  ={\int\limits_{\mathbb{R}^{N}}}\left(  N-\left\vert
x\right\vert ^{2}\right)  \left\vert u\right\vert ^{2}\mathrm{dx}%
+{\int\limits_{\mathbb{R}^{N}}}\left\vert \mathcal{R}u+\left\vert x\right\vert
u\right\vert ^{2}\mathrm{dx}\\
{\int\limits_{\mathbb{R}^{N}}}\left\vert \nabla u\right\vert ^{2}\mathrm{dx}
&  ={\int\limits_{\mathbb{R}^{N}}}\left(  N-\left\vert x\right\vert
^{2}\right)  \left\vert u\right\vert ^{2}\mathrm{dx}+{\int\limits_{\mathbb{R}%
^{N}}}\left\vert \nabla u+xu\right\vert ^{2}\mathrm{dx}.
\end{align*}
These identities imply the scale noninvariant Uncertainty Principle:%
\begin{align*}
&  {\int\limits_{\mathbb{R}^{N}}}\left\vert \nabla u\right\vert ^{2}%
\mathrm{dx}+{\int\limits_{\mathbb{R}^{N}}}\left\vert x\right\vert
^{2}\left\vert u\right\vert ^{2}\mathrm{dx}\\
&  \geq{\int\limits_{\mathbb{R}^{N}}}\left\vert \mathcal{R}u\right\vert
^{2}\mathrm{dx}+{\int\limits_{\mathbb{R}^{N}}}\left\vert x\right\vert
^{2}\left\vert u\right\vert ^{2}\mathrm{dx}\\
&  \geq N{\int\limits_{\mathbb{R}^{N}}}\left\vert u\right\vert ^{2}%
\mathrm{dx}.
\end{align*}

\end{corollary}

\begin{corollary}
\label{c4}If $a+b\neq N-1$, choose $A=sign\left(  N-a-b-1\right)  r^{-b}$,
$B=r^{-a}$. Then $C=\left[  \left\vert N-1-a-b\right\vert r^{-a-b-1}%
-r^{-2a}\right]  $. By the Hardy identities (\ref{H2.1}) and (\ref{H2.2}), for
all $u\in C_{0}^{\infty}\left(  \mathbb{R}^{N}\right)  :$%
\begin{align*}
&  {\int\limits_{B_{R}}}\frac{\left\vert \mathcal{R}u\right\vert ^{2}%
}{\left\vert x\right\vert ^{2b}}\mathrm{dx}+{\int\limits_{B_{R}}}%
\frac{\left\vert u\right\vert ^{2}}{\left\vert x\right\vert ^{2a}}%
\mathrm{dx}-\left\vert N-1-a-b\right\vert {\int\limits_{B_{R}}}\frac
{\left\vert u\right\vert ^{2}}{\left\vert x\right\vert ^{a+b+1}}\mathrm{dx}\\
&  ={\int\limits_{B_{R}}}\left\vert sign\left(  N-a-b-1\right)  \frac
{\mathcal{R}u}{\left\vert x\right\vert ^{b}}+\frac{u}{\left\vert x\right\vert
^{a}}\right\vert ^{2}\mathrm{dx}%
\end{align*}
and%
\begin{align*}
&  {\int\limits_{B_{R}}}\frac{\left\vert \nabla u\right\vert ^{2}}{\left\vert
x\right\vert ^{2b}}\mathrm{dx}+{\int\limits_{B_{R}}}\frac{\left\vert
u\right\vert ^{2}}{\left\vert x\right\vert ^{2a}}\mathrm{dx}-\left\vert
N-1-a-b\right\vert {\int\limits_{B_{R}}}\frac{\left\vert u\right\vert ^{2}%
}{\left\vert x\right\vert ^{a+b+1}}\mathrm{dx}\\
&  ={\int\limits_{B_{R}}}\left\vert sign\left(  N-a-b-1\right)  \frac{\nabla
u}{\left\vert x\right\vert ^{b}}+\frac{u}{\left\vert x\right\vert ^{a}}%
\frac{x}{\left\vert x\right\vert }\right\vert ^{2}\mathrm{dx}.
\end{align*}

\end{corollary}

\begin{corollary}
[Hardy inequalities with Bessel pairs]\label{c5}\textit{Let }$0<R\leq\infty
$\textit{, }$V\geq0$\textit{ and }$W$\textit{ be }$C^{1}$-\textit{functions on
}$\left(  0,R\right)  $. Assume that $\left(  r^{N-1}V,r^{N-1}W\right)  $ is a
Bessel pair on $\left(  0,R\right)  $, that is there exists a positive
function $\varphi$ such that
\[
\left(  r^{N-1}V\varphi^{\prime}\right)  ^{\prime}+r^{N-1}W\varphi=0\text{ on
}\left(  0,R\right)  \text{.}%
\]
In this case, choose $A=\sqrt{V}$, $B=-\frac{\varphi^{\prime}}{\varphi}%
\sqrt{V}$. Then $C=W$. Hence by the Hardy identities (\ref{H2.1}) and
(\ref{H2.2}), we deduce for all $u\in C_{0}^{\infty}\left(  B_{R}%
\setminus\left\{  0\right\}  \right)  :$
\begin{align*}
{\int\limits_{B_{R}}}V\left(  \left\vert x\right\vert \right)  \left\vert
\mathcal{R}u\right\vert ^{2}\mathrm{dx}  &  ={\int\limits_{B_{R}}}W\left(
\left\vert x\right\vert \right)  \left\vert u\right\vert ^{2}\mathrm{dx}%
+{\int\limits_{B_{R}}}\left\vert \sqrt{V\left(  \left\vert x\right\vert
\right)  }\mathcal{R}u-\frac{\varphi^{\prime}\left(  \left\vert x\right\vert
\right)  }{\varphi\left(  \left\vert x\right\vert \right)  }\sqrt{V\left(
\left\vert x\right\vert \right)  }u\right\vert ^{2}\mathrm{dx}\\
&  ={\int\limits_{B_{R}}}W\left(  \left\vert x\right\vert \right)  \left\vert
u\right\vert ^{2}\mathrm{dx}+{\int\limits_{B_{R}}}V\left(  \left\vert
x\right\vert \right)  \varphi^{2}\left(  \left\vert x\right\vert \right)
\left\vert \mathcal{R}\left(  \frac{u\left(  x\right)  }{\varphi\left(
\left\vert x\right\vert \right)  }\right)  \right\vert ^{2}\mathrm{dx}%
\end{align*}
and%
\begin{align*}
{\int\limits_{B_{R}}}V\left(  \left\vert x\right\vert \right)  \left\vert
\nabla u\right\vert ^{2}\mathrm{dx}  &  ={\int\limits_{B_{R}}}W\left(
\left\vert x\right\vert \right)  \left\vert u\right\vert ^{2}\mathrm{dx}%
+{\int\limits_{B_{R}}}\left\vert \sqrt{V\left(  \left\vert x\right\vert
\right)  }\nabla u-\frac{\varphi^{\prime}\left(  \left\vert x\right\vert
\right)  }{\varphi\left(  \left\vert x\right\vert \right)  }\sqrt{V\left(
\left\vert x\right\vert \right)  }u\frac{x}{\left\vert x\right\vert
}\right\vert ^{2}\mathrm{dx}\\
&  ={\int\limits_{B_{R}}}W\left(  \left\vert x\right\vert \right)  \left\vert
u\right\vert ^{2}\mathrm{dx}+{\int\limits_{B_{R}}}V\left(  \left\vert
x\right\vert \right)  \varphi^{2}\left(  \left\vert x\right\vert \right)
\left\vert \nabla\left(  \frac{u\left(  x\right)  }{\varphi\left(  \left\vert
x\right\vert \right)  }\right)  \right\vert ^{2}\mathrm{dx}.
\end{align*}

\end{corollary}

We note that Hardy type inequalities and identities have been investigated by
many researchers. See, \cite{DLL22, GM1, FLL21}, for instance. We also refer
the reader to the papers \cite{LLZ19, LLZ20} in which the Hardy type
inequalities and identities have been studied for more general distance functions.

\subsection{$L^{2}$-Caffarelli-Kohn-Nirenberg identities}

By optimizing $\alpha$ in the Lemma \ref{key}, that is by choosing
\[
\alpha\left(  {\int\limits_{B_{R}}}A^{2}\left(  \left\vert x\right\vert
\right)  \left\vert \mathcal{R}u\right\vert ^{2}\mathrm{dx}\right)  ^{\frac
{1}{2}}=\frac{1}{\alpha}\left(  {\int\limits_{B_{R}}}B^{2}\left(  \left\vert
x\right\vert \right)  \left\vert u\right\vert ^{2}\mathrm{dx}\right)
^{\frac{1}{2}}%
\]
in (\ref{CKN2.1}) and
\[
\alpha\left(  {\int\limits_{B_{R}}}A^{2}\left(  \left\vert x\right\vert
\right)  \left\vert \nabla u\right\vert ^{2}\mathrm{dx}\right)  ^{\frac{1}{2}%
}=\frac{1}{\alpha}\left(  {\int\limits_{B_{R}}}B^{2}\left(  \left\vert
x\right\vert \right)  \left\vert u\right\vert ^{2}\mathrm{dx}\right)
^{\frac{1}{2}}%
\]
in (\ref{CKN2.2}), we have the following Caffarelli-Kohn-Nirenberg identities

\begin{theorem}
\label{T2}\textit{Let }$0<R\leq\infty$\textit{, }$A$\textit{ and }$B$\textit{
be }$C^{1}$-\textit{functions on }$\left(  0,R\right)  $ and let
\[
C\left(  r\right)  =\left(  A\left(  r\right)  B\left(  r\right)  \right)
^{\prime}+\left(  N-1\right)  \frac{A\left(  r\right)  B\left(  r\right)  }%
{r}-B^{2}\left(  r\right)  \text{.}%
\]
Then for all $u\in C_{0}^{\infty}\left(  B_{R}\setminus\left\{  0\right\}
\right)  $, we have%
\begin{align*}
&  \left(  {\int\limits_{B_{R}}}A^{2}\left(  \left\vert x\right\vert \right)
\left\vert \mathcal{R}u\right\vert ^{2}\mathrm{dx}\right)  ^{\frac{1}{2}%
}\left(  {\int\limits_{B_{R}}}B^{2}\left(  \left\vert x\right\vert \right)
\left\vert u\right\vert ^{2}\mathrm{dx}\right)  ^{\frac{1}{2}}\\
&  =\frac{1}{2}{\int\limits_{B_{R}}}\left[  C\left(  \left\vert x\right\vert
\right)  +B^{2}\left(  \left\vert x\right\vert \right)  \right]  \left\vert
u\right\vert ^{2}\mathrm{dx}\\
&  +\frac{1}{2}{\int\limits_{B_{R}}}\left\vert \frac{\left\Vert Bu\right\Vert
_{2}^{\frac{1}{2}}}{\left\Vert A\mathcal{R}u\right\Vert _{2}^{\frac{1}{2}}%
}A\left(  \left\vert x\right\vert \right)  \mathcal{R}u+\frac{\left\Vert
A\mathcal{R}u\right\Vert _{2}^{\frac{1}{2}}}{\left\Vert Bu\right\Vert
_{2}^{\frac{1}{2}}}B\left(  \left\vert x\right\vert \right)  u\right\vert
^{2}\mathrm{dx}%
\end{align*}
and%
\begin{align*}
&  \left(  {\int\limits_{B_{R}}}A^{2}\left(  \left\vert x\right\vert \right)
\left\vert \nabla u\right\vert ^{2}\mathrm{dx}\right)  ^{\frac{1}{2}}\left(
{\int\limits_{B_{R}}}B^{2}\left(  \left\vert x\right\vert \right)  \left\vert
u\right\vert ^{2}\mathrm{dx}\right)  ^{\frac{1}{2}}\\
&  =\frac{1}{2}{\int\limits_{B_{R}}}\left[  C\left(  \left\vert x\right\vert
\right)  +B^{2}\left(  \left\vert x\right\vert \right)  \right]  \left\vert
u\right\vert ^{2}\mathrm{dx}\\
&  +\frac{1}{2}{\int\limits_{B_{R}}}\left\vert \frac{\left\Vert Bu\right\Vert
_{2}^{\frac{1}{2}}}{\left\Vert A\left\vert \nabla u\right\vert \right\Vert
_{2}^{\frac{1}{2}}}A\left(  \left\vert x\right\vert \right)  \nabla
u+\frac{\left\Vert A\left\vert \nabla u\right\vert \right\Vert _{2}^{\frac
{1}{2}}}{\left\Vert Bu\right\Vert _{2}^{\frac{1}{2}}}B\left(  \left\vert
x\right\vert \right)  u\frac{x}{\left\vert x\right\vert }\right\vert
^{2}\mathrm{dx}.
\end{align*}

\end{theorem}

We provide here some examples:

\begin{corollary}
\label{c6}Choose $A=1$, $B=r$. Then $H=r$, $H^{\prime}\left(  r\right)
+\left(  N-1\right)  \frac{H\left(  r\right)  }{r}=N$ and $C=N-r^{2}$. From
the Caffarelli-Kohn-Nirenberg identities, we have the scale invariant
Heisenberg Uncertainty Principles:
\begin{align*}
&  \left(  {\int\limits_{\mathbb{R}^{N}}}\left\vert \mathcal{R}u\right\vert
^{2}\mathrm{dx}\right)  ^{\frac{1}{2}}\left(  {\int\limits_{\mathbb{R}^{N}}%
}\left\vert x\right\vert ^{2}\left\vert u\right\vert ^{2}\mathrm{dx}\right)
^{\frac{1}{2}}-\frac{N}{2}{\int\limits_{\mathbb{R}^{N}}}\left\vert
u\right\vert ^{2}\mathrm{dx}\\
&  =\frac{1}{2}{\int\limits_{\mathbb{R}^{N}}}\left\vert \frac{\left\Vert
\left\vert x\right\vert u\right\Vert _{2}^{\frac{1}{2}}}{\left\Vert
\mathcal{R}u\right\Vert _{2}^{\frac{1}{2}}}\mathcal{R}u+\frac{\left\Vert
\mathcal{R}u\right\Vert _{2}^{\frac{1}{2}}}{\left\Vert \left\vert x\right\vert
u\right\Vert _{2}^{\frac{1}{2}}}\left\vert x\right\vert u\right\vert
^{2}\mathrm{dx}%
\end{align*}
and
\begin{align*}
&  \left(  {\int\limits_{\mathbb{R}^{N}}}\left\vert \nabla u\right\vert
^{2}\mathrm{dx}\right)  ^{\frac{1}{2}}\left(  {\int\limits_{\mathbb{R}^{N}}%
}\left\vert x\right\vert ^{2}\left\vert u\right\vert ^{2}\mathrm{dx}\right)
^{\frac{1}{2}}-\frac{N}{2}{\int\limits_{\mathbb{R}^{N}}}\left\vert
u\right\vert ^{2}\mathrm{dx}\\
&  =\frac{1}{2}{\int\limits_{\mathbb{R}^{N}}}\left\vert \frac{\left\Vert
\left\vert x\right\vert u\right\Vert _{2}^{\frac{1}{2}}}{\left\Vert \nabla
u\right\Vert _{2}^{\frac{1}{2}}}\nabla u+\frac{\left\Vert \nabla u\right\Vert
_{2}^{\frac{1}{2}}}{\left\Vert \left\vert x\right\vert u\right\Vert
_{2}^{\frac{1}{2}}}xu\right\vert ^{2}\mathrm{dx}\\
&  =\frac{\lambda^{2}}{2}{\int\limits_{\mathbb{R}^{N}}}\left\vert
\nabla\left(  ue^{\frac{\left\vert x\right\vert ^{2}}{2\lambda^{2}}}\right)
\right\vert ^{2}e^{-\frac{\left\vert x\right\vert ^{2}}{\lambda^{2}}%
}\mathrm{dx}\text{.}%
\end{align*}
Here $\lambda=\frac{\left(  {\int\limits_{\mathbb{R}^{N}}}\left\vert
x\right\vert ^{2}\left\vert u\right\vert ^{2}\mathrm{dx}\right)  ^{\frac{1}%
{4}}}{\left(  {\int\limits_{\mathbb{R}^{N}}}\left\vert \nabla u\right\vert
^{2}\mathrm{dx}\right)  ^{\frac{1}{4}}}$.
\end{corollary}

\begin{corollary}
\label{c7}If $a+b\neq N-1$, choose $A=sign\left(  N-a-b-1\right)  r^{-b}$,
$B=r^{-a}$. Then $C=\left[  \left\vert N-1-a-b\right\vert r^{-a-b-1}%
-r^{-2a}\right]  $. From the Caffarelli-Kohn-Nirenberg identities, we have
\begin{align*}
&  \left(  {\int\limits_{\mathbb{R}^{N}}}\frac{1}{\left\vert x\right\vert
^{2b}}\left\vert \mathcal{R}u\right\vert ^{2}\mathrm{dx}\right)  ^{\frac{1}%
{2}}\left(  {\int\limits_{\mathbb{R}^{N}}}\frac{1}{\left\vert x\right\vert
^{2a}}\left\vert u\right\vert ^{2}\mathrm{dx}\right)  ^{\frac{1}{2}}\\
&  =\left\vert \frac{N-1-a-b}{2}\right\vert {\int\limits_{\mathbb{R}^{N}}%
}\frac{1}{\left\vert x\right\vert ^{a+b+1}}\left\vert u\right\vert
^{2}\mathrm{dx}\\
&  +\frac{1}{2}{\int\limits_{\mathbb{R}^{N}}}\left\vert sign\left(
N-a-b-1\right)  \frac{\left\Vert \frac{u}{\left\vert x\right\vert ^{a}%
}\right\Vert _{2}^{\frac{1}{2}}}{\left\Vert \frac{\mathcal{R}u}{\left\vert
x\right\vert ^{b}}\right\Vert _{2}^{\frac{1}{2}}}\frac{1}{\left\vert
x\right\vert ^{b}}\mathcal{R}u+\frac{\left\Vert \frac{\mathcal{R}u}{\left\vert
x\right\vert ^{b}}\right\Vert _{2}^{\frac{1}{2}}}{\left\Vert \frac
{u}{\left\vert x\right\vert ^{a}}\right\Vert _{2}^{\frac{1}{2}}}\frac
{1}{\left\vert x\right\vert ^{a}}u\right\vert ^{2}\mathrm{dx}%
\end{align*}
and
\begin{align*}
&  \left(  {\int\limits_{\mathbb{R}^{N}}}\frac{1}{\left\vert x\right\vert
^{2b}}\left\vert \nabla u\right\vert ^{2}\mathrm{dx}\right)  ^{\frac{1}{2}%
}\left(  {\int\limits_{\mathbb{R}^{N}}}\frac{1}{\left\vert x\right\vert ^{2a}%
}\left\vert u\right\vert ^{2}\mathrm{dx}\right)  ^{\frac{1}{2}}\\
&  =\left\vert \frac{N-1-a-b}{2}\right\vert {\int\limits_{\mathbb{R}^{N}}%
}\frac{1}{\left\vert x\right\vert ^{a+b+1}}\left\vert u\right\vert
^{2}\mathrm{dx}\\
&  +\frac{1}{2}{\int\limits_{\mathbb{R}^{N}}}\left\vert sign\left(
N-a-b-1\right)  \frac{\left\Vert \frac{u}{\left\vert x\right\vert ^{a}%
}\right\Vert _{2}^{\frac{1}{2}}}{\left\Vert \frac{\left\vert \nabla
u\right\vert }{\left\vert x\right\vert ^{b}}\right\Vert _{2}^{\frac{1}{2}}%
}\frac{1}{\left\vert x\right\vert ^{b}}\nabla u+\frac{\left\Vert
\frac{\left\vert \nabla u\right\vert }{\left\vert x\right\vert ^{b}%
}\right\Vert _{2}^{\frac{1}{2}}}{\left\Vert \frac{u}{\left\vert x\right\vert
^{a}}\right\Vert _{2}^{\frac{1}{2}}}\frac{1}{\left\vert x\right\vert ^{a}%
}u\frac{x}{\left\vert x\right\vert }\right\vert ^{2}\mathrm{dx}.
\end{align*}
These identities imply
\begin{align*}
\left(  {\int\limits_{\mathbb{R}^{N}}}\frac{1}{\left\vert x\right\vert ^{2b}%
}\left\vert \nabla u\right\vert ^{2}\mathrm{dx}\right)  \left(  {\int
\limits_{\mathbb{R}^{N}}}\frac{1}{\left\vert x\right\vert ^{2a}}\left\vert
u\right\vert ^{2}\mathrm{dx}\right)   &  \geq\left(  {\int\limits_{\mathbb{R}%
^{N}}}\frac{1}{\left\vert x\right\vert ^{2b}}\left\vert \mathcal{R}%
u\right\vert ^{2}\mathrm{dx}\right)  \left(  {\int\limits_{\mathbb{R}^{N}}%
}\frac{1}{\left\vert x\right\vert ^{2a}}\left\vert u\right\vert ^{2}%
\mathrm{dx}\right) \\
&  \geq\left\vert \frac{N-1-a-b}{2}\right\vert ^{2}\left(  {\int
\limits_{\mathbb{R}^{N}}}\frac{1}{\left\vert x\right\vert ^{a+b+1}}\left\vert
u\right\vert ^{2}\mathrm{dx}\right)  ^{2}.
\end{align*}

\end{corollary}

\begin{corollary}
[Caffarelli-Kohn-Nirenberg inequalities with Bessel pairs]\label{c8}%
\textit{Let }$0<R\leq\infty$\textit{, }$V$\textit{ and }$W$\textit{ be
positive }$C^{1}-$\textit{functions on }$\left(  0,R\right)  $. If $\left(
r^{N-1}V,r^{N-1}W\right)  $ is a Bessel pair on $\left(  0,R\right)  $, that
is there exists $\varphi>0$ on $\left(  0,R\right)  $ such that
\[
\left(  r^{N-1}V\varphi^{\prime}\right)  ^{\prime}+r^{N-1}W\varphi=0.
\]
Then as above, we choose $A=\sqrt{V}$, $B=-\frac{\varphi^{\prime}}{\varphi
}\sqrt{V}$, $C=W$. Then, we have from the Caffarelli-Kohn-Nirenberg identities
that
\begin{align*}
&  \left(  {\int\limits_{B_{R}}}V\left(  \left\vert x\right\vert \right)
\left\vert \mathcal{R}u\right\vert ^{2}\mathrm{dx}\right)  ^{\frac{1}{2}%
}\left(  {\int\limits_{B_{R}}}\left(  \frac{\varphi^{\prime}\left(  \left\vert
x\right\vert \right)  }{\varphi\left(  \left\vert x\right\vert \right)
}\right)  ^{2}V\left(  \left\vert x\right\vert \right)  \left\vert
u\right\vert ^{2}\mathrm{dx}\right)  ^{\frac{1}{2}}\\
&  =\frac{1}{2}{\int\limits_{B_{R}}}\left[  W\left(  \left\vert x\right\vert
\right)  +\left(  \frac{\varphi^{\prime}\left(  \left\vert x\right\vert
\right)  }{\varphi\left(  \left\vert x\right\vert \right)  }\right)
^{2}V\left(  \left\vert x\right\vert \right)  \right]  \left\vert u\right\vert
^{2}\mathrm{dx}\\
&  +\frac{1}{2}{\int\limits_{B_{R}}}\left\vert \frac{\left\Vert \frac
{\varphi^{\prime}}{\varphi}\sqrt{V}u\right\Vert _{2}^{\frac{1}{2}}}{\left\Vert
\sqrt{V}\mathcal{R}u\right\Vert _{2}^{\frac{1}{2}}}\sqrt{V\left(  \left\vert
x\right\vert \right)  }\mathcal{R}u-\frac{\left\Vert \sqrt{V}\mathcal{R}%
u\right\Vert _{2}^{\frac{1}{2}}}{\left\Vert \frac{\varphi^{\prime}}{\varphi
}\sqrt{V}u\right\Vert _{2}^{\frac{1}{2}}}\frac{\varphi^{\prime}\left(
\left\vert x\right\vert \right)  }{\varphi\left(  \left\vert x\right\vert
\right)  }\sqrt{V\left(  \left\vert x\right\vert \right)  }u\right\vert
^{2}\mathrm{dx}%
\end{align*}
and%
\begin{align*}
&  \left(  {\int\limits_{B_{R}}}V\left(  \left\vert x\right\vert \right)
\left\vert \nabla u\right\vert ^{2}\mathrm{dx}\right)  ^{\frac{1}{2}}\left(
{\int\limits_{B_{R}}}\left(  \frac{\varphi^{\prime}\left(  \left\vert
x\right\vert \right)  }{\varphi\left(  \left\vert x\right\vert \right)
}\right)  ^{2}V\left(  \left\vert x\right\vert \right)  \left\vert
u\right\vert ^{2}\mathrm{dx}\right)  ^{\frac{1}{2}}\\
&  =\frac{1}{2}{\int\limits_{B_{R}}}\left[  W\left(  \left\vert x\right\vert
\right)  +\left(  \frac{\varphi^{\prime}\left(  \left\vert x\right\vert
\right)  }{\varphi\left(  \left\vert x\right\vert \right)  }\right)
^{2}V\left(  \left\vert x\right\vert \right)  \right]  \left\vert u\right\vert
^{2}\mathrm{dx}\\
&  +\frac{1}{2}{\int\limits_{B_{R}}}\left\vert \frac{\left\Vert \frac
{\varphi^{\prime}}{\varphi}\sqrt{V}u\right\Vert _{2}^{\frac{1}{2}}}{\left\Vert
\sqrt{V}\left\vert \nabla u\right\vert \right\Vert _{2}^{\frac{1}{2}}}%
\sqrt{V\left(  \left\vert x\right\vert \right)  }\nabla u-\frac{\left\Vert
\sqrt{V}\left\vert \nabla u\right\vert \right\Vert _{2}^{\frac{1}{2}}%
}{\left\Vert \frac{\varphi^{\prime}}{\varphi}\sqrt{V}u\right\Vert _{2}%
^{\frac{1}{2}}}\frac{\varphi^{\prime}\left(  \left\vert x\right\vert \right)
}{\varphi\left(  \left\vert x\right\vert \right)  }\sqrt{V\left(  \left\vert
x\right\vert \right)  }u\frac{x}{\left\vert x\right\vert }\right\vert
^{2}\mathrm{dx}.
\end{align*}
Hence, we deduce the following family of Caffarelli-Kohn-Nirenberg
inequalities with Bessel pairs%
\begin{align*}
&  \left(  {\int\limits_{B_{R}}}V\left(  \left\vert x\right\vert \right)
\left\vert \nabla u\right\vert ^{2}\mathrm{dx}\right)  ^{\frac{1}{2}}\left(
{\int\limits_{B_{R}}}\left(  \frac{\varphi^{\prime}\left(  \left\vert
x\right\vert \right)  }{\varphi\left(  \left\vert x\right\vert \right)
}\right)  ^{2}V\left(  \left\vert x\right\vert \right)  \left\vert
u\right\vert ^{2}\mathrm{dx}\right)  ^{\frac{1}{2}}\\
&  \geq\left(  {\int\limits_{B_{R}}}V\left(  \left\vert x\right\vert \right)
\left\vert \mathcal{R}u\right\vert ^{2}\mathrm{dx}\right)  ^{\frac{1}{2}%
}\left(  {\int\limits_{B_{R}}}\left(  \frac{\varphi^{\prime}\left(  \left\vert
x\right\vert \right)  }{\varphi\left(  \left\vert x\right\vert \right)
}\right)  ^{2}V\left(  \left\vert x\right\vert \right)  \left\vert
u\right\vert ^{2}\mathrm{dx}\right)  ^{\frac{1}{2}}\\
&  \geq\frac{1}{2}{\int\limits_{B_{R}}}\left[  W\left(  \left\vert
x\right\vert \right)  +\left(  \frac{\varphi^{\prime}\left(  \left\vert
x\right\vert \right)  }{\varphi\left(  \left\vert x\right\vert \right)
}\right)  ^{2}V\left(  \left\vert x\right\vert \right)  \right]  \left\vert
u\right\vert ^{2}\mathrm{dx}.
\end{align*}

\end{corollary}

\section{The stability of the Caffarelli-Kohn-Nirenberg inequalities}

In this section, we will use the Caffarelli-Kohn-Nirenberg identities derived
in the previous section and the Poincar\'{e} inequality to investigate the
stability of the Caffarelli-Kohn-Nirenberg inequalities. To illustriate our
approach, we will start with the stability of the Heisenberg Uncertainty
Principle, which is a special case of the Caffarelli-Kohn-Nirenberg inequalities.

\subsection{The stability of the Heisenberg Uncertainty Principle}

Recall that from the Corollary \ref{c1}, we have that for $u\in X:=W^{1,2}%
\left(  \mathbb{R}^{N}\right)  \cap\left\{  u:{\int\limits_{\mathbb{R}^{N}}%
}\left\vert x\right\vert ^{2}\left\vert u\right\vert ^{2}\mathrm{dx}%
<\infty\right\}  $,
\begin{align}
&  {\int\limits_{\mathbb{R}^{N}}}\left\vert \nabla u\right\vert ^{2}%
\mathrm{dx}+{\int\limits_{\mathbb{R}^{N}}}\left\vert x\right\vert
^{2}\left\vert u\right\vert ^{2}\mathrm{dx}-N{\int\limits_{\mathbb{R}^{N}}%
}\left\vert u\right\vert ^{2}\mathrm{dx}\nonumber\\
&  ={\int\limits_{\mathbb{R}^{N}}}\left\vert \nabla\left(  ue^{\frac{1}%
{2}\left\vert x\right\vert ^{2}}\right)  \right\vert ^{2}e^{-\left\vert
x\right\vert ^{2}}\mathrm{dx}. \label{HUP1}%
\end{align}
Our goal is to apply the Poincar\'{e} inequality to the right hand side of
(\ref{HUP1}). Therefore, we will need the following classical Poincar\'{e}
inequality for Gaussian measure (see \cite{BGL14, Bec89}, for instance):

\begin{lemma}
\label{l1}For all smooth function $v:$%
\[
{\int\limits_{\mathbb{R}^{N}}}\left\vert \nabla v\right\vert ^{2}\left(
2\pi\right)  ^{-\frac{N}{2}}e^{-\frac{1}{2}\left\vert x\right\vert ^{2}%
}\mathrm{dx}\geq{\int\limits_{\mathbb{R}^{N}}}\left\vert v-\left(
2\pi\right)  ^{-\frac{N}{2}}{\int\limits_{\mathbb{R}^{N}}v}e^{-\frac{1}%
{2}\left\vert x\right\vert ^{2}}\mathrm{dx}\right\vert ^{2}\left(
2\pi\right)  ^{-\frac{N}{2}}e^{-\frac{1}{2}\left\vert x\right\vert ^{2}%
}\mathrm{dx}\text{.}%
\]

\end{lemma}

Obviously, the above classical Poincar\'{e} inequality for Gaussian measure
implies the following Poincar\'{e} inequality for Gaussian type measure

\begin{lemma}
Let $\lambda>0$. For all smooth function $v:$%
\begin{equation}
{\int\limits_{\mathbb{R}^{N}}}\left\vert \nabla v\right\vert ^{2}%
e^{-\frac{\left\vert x\right\vert ^{2}}{\lambda^{2}}}\mathrm{dx}\geq\frac
{2}{\lambda^{2}}\inf_{c}{\int\limits_{\mathbb{R}^{N}}}\left\vert
v-c\right\vert ^{2}e^{-\frac{\left\vert x\right\vert ^{2}}{\lambda^{2}}%
}\mathrm{dx}\text{.} \label{HUP2}%
\end{equation}

\end{lemma}

Hence, by combining (\ref{HUP1}) and (\ref{HUP2}), we get
\begin{align*}
&  {\int\limits_{\mathbb{R}^{N}}}\left\vert \nabla u\right\vert ^{2}%
\mathrm{dx}+{\int\limits_{\mathbb{R}^{N}}}\left\vert x\right\vert
^{2}\left\vert u\right\vert ^{2}\mathrm{dx}-N{\int\limits_{\mathbb{R}^{N}}%
}\left\vert u\right\vert ^{2}\mathrm{dx}\\
&  ={\int\limits_{\mathbb{R}^{N}}}\left\vert \nabla\left(  ue^{\frac{1}%
{2}\left\vert x\right\vert ^{2}}\right)  \right\vert ^{2}e^{-\left\vert
x\right\vert ^{2}}\mathrm{dx}\\
&  \geq2\inf_{c}{\int\limits_{\mathbb{R}^{N}}}\left\vert ue^{\frac{1}%
{2}\left\vert x\right\vert ^{2}}-c\right\vert ^{2}e^{-\left\vert x\right\vert
^{2}}\mathrm{dx}\\
&  =2\inf_{c}{\int\limits_{\mathbb{R}^{N}}}\left\vert u-ce^{-\frac{1}%
{2}\left\vert x\right\vert ^{2}}\right\vert ^{2}\mathrm{dx}.
\end{align*}
Thus, we obtain the following stability version of the scale non-invariant
Heisenberg Uncertainty Principle

\begin{theorem}
\label{T3.1}For $u\in X$, then
\begin{align*}
&  {\int\limits_{\mathbb{R}^{N}}}\left\vert \nabla u\right\vert ^{2}%
\mathrm{dx}+{\int\limits_{\mathbb{R}^{N}}}\left\vert x\right\vert
^{2}\left\vert u\right\vert ^{2}\mathrm{dx}-N{\int\limits_{\mathbb{R}^{N}}%
}\left\vert u\right\vert ^{2}\mathrm{dx}\\
&  \geq2\inf_{c}{\int\limits_{\mathbb{R}^{N}}}\left\vert u-ce^{-\frac{1}%
{2}\left\vert x\right\vert ^{2}}\right\vert ^{2}\mathrm{dx}.
\end{align*}
The constant $2$ is sharp and can be achieved by nontrivial functions.
\end{theorem}

\begin{proof}
Let $u=x_{1}e^{-\frac{1}{2}\left\vert x\right\vert ^{2}}$. Then by direct
computations, we get%
\[
{\int\limits_{\mathbb{R}^{N}}}\left\vert \nabla u\right\vert ^{2}%
\mathrm{dx}+{\int\limits_{\mathbb{R}^{N}}}\left\vert x\right\vert
^{2}\left\vert u\right\vert ^{2}\mathrm{dx}-N{\int\limits_{\mathbb{R}^{N}}%
}\left\vert u\right\vert ^{2}\mathrm{dx}=\pi^{\frac{N}{2}}%
\]
and
\begin{align*}
&  \inf_{c}{\int\limits_{\mathbb{R}^{N}}}\left\vert u-ce^{-\frac{1}%
{2}\left\vert x\right\vert ^{2}}\right\vert ^{2}\mathrm{dx}\\
&  =\inf_{c}\left(  \frac{\pi^{\frac{N}{2}}}{2}+c^{2}\pi^{\frac{N}{2}}\right)
\\
&  =\frac{\pi^{\frac{N}{2}}}{2}.
\end{align*}
Therefore,
\[
{\int\limits_{\mathbb{R}^{N}}}\left\vert \nabla u\right\vert ^{2}%
\mathrm{dx}+{\int\limits_{\mathbb{R}^{N}}}\left\vert x\right\vert
^{2}\left\vert u\right\vert ^{2}\mathrm{dx}-N{\int\limits_{\mathbb{R}^{N}}%
}\left\vert u\right\vert ^{2}\mathrm{dx}=2\inf_{c}{\int\limits_{\mathbb{R}%
^{N}}}\left\vert u-ce^{-\frac{1}{2}\left\vert x\right\vert ^{2}}\right\vert
^{2}\mathrm{dx}\text{.}%
\]

\end{proof}

Now, let $E=\left\{  \alpha e^{-\frac{\beta}{2}\left\vert x\right\vert ^{2}%
}\text{, }\alpha\in\mathbb{R}\text{, }\beta>0\right\}  $ be the manifold of
optimizers of the scale invariant Heisenberg Uncertainty Principle%
\[
\left(  {\int\limits_{\mathbb{R}^{N}}}\left\vert \nabla u\right\vert
^{2}\mathrm{dx}\right)  \left(  {\int\limits_{\mathbb{R}^{N}}}\left\vert
x\right\vert ^{2}\left\vert u\right\vert ^{2}\mathrm{dx}\right)  \geq
\frac{N^{2}}{4}\left(  {\int\limits_{\mathbb{R}^{N}}}\left\vert u\right\vert
^{2}\mathrm{dx}\right)  ^{2}\text{.}%
\]
It was showed in \cite{McV21} that $E$ forms a closed cone in $L^{2}\left(
\mathbb{R}^{N}\right)  $. In particular, for all $u\in X$, there exists a
$u^{\ast}\in E$ such that
\[
\inf_{v\in E}\left\Vert u-v\right\Vert _{2}=\left\Vert u-u^{\ast}\right\Vert
_{2}\text{.}%
\]

We also recall the Heisenberg deficit
\[
\delta_{1}\left(  u\right)  =\left(  {\int\limits_{\mathbb{R}^{N}}}\left\vert
\nabla u\right\vert ^{2}\mathrm{dx}\right)  ^{\frac{1}{2}}\left(
{\int\limits_{\mathbb{R}^{N}}}\left\vert x\right\vert ^{2}\left\vert
u\right\vert ^{2}\mathrm{dx}\right)  ^{\frac{1}{2}}-\frac{N}{2}{\int
\limits_{\mathbb{R}^{N}}}\left\vert u\right\vert ^{2}\mathrm{dx}.
\]

Then, by using the Corollary \ref{c6}, we get the quantitative stability of
the scale invariant Heisenberg Uncertainty Principle, namely, Theorem \ref{E2}:

\begin{theorem}
\label{T3.2}For all $u\in X:$%
\[
\delta_{1}\left(  u\right)  \geq\inf_{u^{\ast}\in E}\left\Vert u-u^{\ast
}\right\Vert _{2}^{2}.
\]
Moreover, the inequality is sharp and the equality can be attained by
nontrivial functions.
\end{theorem}

\begin{proof}
Let $u\in X\setminus\left\{  0\right\}  $. By Corollary \ref{c6}, we obtain
from the Poincar\'{e} inequality for Gaussian type measure (\ref{HUP2}) with
$\lambda=\left(  \frac{{\int\limits_{\mathbb{R}^{N}}}\left\vert x\right\vert
^{2}\left\vert u\right\vert ^{2}\mathrm{dx}}{{\int\limits_{\mathbb{R}^{N}}%
}\left\vert \nabla u\right\vert ^{2}\mathrm{dx}}\right)  ^{\frac{1}{4}}$ that%
\begin{align*}
&  \left(  {\int\limits_{\mathbb{R}^{N}}}\left\vert \nabla u\right\vert
^{2}\mathrm{dx}\right)  ^{\frac{1}{2}}\left(  {\int\limits_{\mathbb{R}^{N}}%
}\left\vert x\right\vert ^{2}\left\vert u\right\vert ^{2}\mathrm{dx}\right)
^{\frac{1}{2}}-\frac{N}{2}{\int\limits_{\mathbb{R}^{N}}}\left\vert
u\right\vert ^{2}\mathrm{dx}\\
&  =\frac{\lambda^{2}}{2}{\int\limits_{\mathbb{R}^{N}}}\left\vert
\nabla\left(  ue^{\frac{\left\vert x\right\vert ^{2}}{2\lambda^{2}}}\right)
\right\vert ^{2}e^{-\frac{\left\vert x\right\vert ^{2}}{\lambda^{2}}%
}\mathrm{dx}\\
&  \geq\inf_{c}{\int\limits_{\mathbb{R}^{N}}}\left\vert u-ce^{-\frac
{1}{2\lambda^{2}}\left\vert x\right\vert ^{2}}\right\vert ^{2}\mathrm{dx}\\
&  \geq\inf_{u^{\ast}\in E}\left\Vert u-u^{\ast}\right\Vert _{2}^{2}.
\end{align*}

Now, let $u=x_{1}e^{-\frac{1}{2}\left\vert x\right\vert ^{2}}$. Then by direct
computations:
\[
\nabla u=\left\langle 1-x_{1}^{2},-x_{1}x_{2},...,-x_{1}x_{N}\right\rangle
e^{-\frac{1}{2}\left\vert x\right\vert ^{2}},
\]
and%
\[
{\int\limits_{\mathbb{R}^{N}}}\left\vert \nabla u\right\vert ^{2}%
\mathrm{dx}={\int\limits_{\mathbb{R}^{N}}}\left\vert x\right\vert
^{2}\left\vert u\right\vert ^{2}\mathrm{dx}=\frac{N+2}{4}\pi^{\frac{N}{2}}.
\]
Therefore,%
\begin{align*}
&  \left(  {\int\limits_{\mathbb{R}^{N}}}\left\vert \nabla u\right\vert
^{2}\mathrm{dx}\right)  ^{\frac{1}{2}}\left(  {\int\limits_{\mathbb{R}^{N}}%
}\left\vert x\right\vert ^{2}\left\vert u\right\vert ^{2}\mathrm{dx}\right)
^{\frac{1}{2}}-\frac{N}{2}{\int\limits_{\mathbb{R}^{N}}}\left\vert
u\right\vert ^{2}\mathrm{dx}\\
&  =\frac{1}{2}\pi^{\frac{N}{2}}\text{.}%
\end{align*}
Also,
\[
{\int\limits_{\mathbb{R}^{N}}}\left\vert u-ce^{-\frac{1}{2\lambda^{2}%
}\left\vert x\right\vert ^{2}}\right\vert ^{2}\mathrm{dx}=\frac{1}{2}%
\pi^{\frac{N}{2}}+\left\vert c\right\vert ^{2}\lambda^{N}\pi^{\frac{N}{2}}.
\]
That is
\[
\inf_{u^{\ast}\in E}\left\Vert u-u^{\ast}\right\Vert _{2}^{2}=\frac{1}{2}%
\pi^{\frac{N}{2}}.
\]
Therefore,%
\[
\left(  {\int\limits_{\mathbb{R}^{N}}}\left\vert \nabla u\right\vert
^{2}\mathrm{dx}\right)  ^{\frac{1}{2}}\left(  {\int\limits_{\mathbb{R}^{N}}%
}\left\vert x\right\vert ^{2}\left\vert u\right\vert ^{2}\mathrm{dx}\right)
^{\frac{1}{2}}-\frac{N}{2}{\int\limits_{\mathbb{R}^{N}}}\left\vert
u\right\vert ^{2}\mathrm{dx}=\inf_{u^{\ast}\in E}\left\Vert u-u^{\ast
}\right\Vert _{2}^{2}.
\]

\end{proof}

A more careful analysis leads to the following version of the quantitative
stability of the scale invariant Heisenberg Uncertainty Principle, namely,
Theorem \ref{E2}:

\begin{theorem}
\label{T3.3}For all $u\in X:$%
\[
\delta_{1}\left(  u\right)  \geq\frac{1}{2}\inf_{u^{\ast}\in E}\left\{
{\int\limits_{\mathbb{R}^{N}}}\left\vert u-u^{\ast}\right\vert ^{2}%
\mathrm{dx}:{\int\limits_{\mathbb{R}^{N}}}\left\vert u\right\vert
^{2}\mathrm{dx}={\int\limits_{\mathbb{R}^{N}}}\left\vert u^{\ast}\right\vert
^{2}\mathrm{dx}\right\}  .
\]
Moreover, the inequality is sharp and the equality can be attained by
nontrivial functions.
\end{theorem}

\begin{proof}
WLOG, we can asssume that ${\int\limits_{\mathbb{R}^{N}}}\left\vert
u\right\vert ^{2}\mathrm{dx}=1$. Now, if
\[
\delta_{1}\left(  u\right)  =\left(  {\int\limits_{\mathbb{R}^{N}}}\left\vert
\nabla u\right\vert ^{2}\mathrm{dx}\right)  ^{1/2}\left(  {\int
\limits_{\mathbb{R}^{N}}}\left\vert x\right\vert ^{2}\left\vert u\right\vert
^{2}\mathrm{dx}\right)  ^{1/2}-\frac{N}{2}{\int\limits_{\mathbb{R}^{N}}%
}\left\vert u\right\vert ^{2}\mathrm{dx}<1,
\]

then by Theorem \ref{T3.2}, we have
\begin{align*}
&  \left(  {\int\limits_{\mathbb{R}^{N}}}\left\vert \nabla u\right\vert
^{2}\mathrm{dx}\right)  ^{1/2}\left(  {\int\limits_{\mathbb{R}^{N}}}\left\vert
x\right\vert ^{2}\left\vert u\right\vert ^{2}\mathrm{dx}\right)  ^{1/2}%
-\frac{N}{2}{\int\limits_{\mathbb{R}^{N}}}\left\vert u\right\vert
^{2}\mathrm{dx}\\
&  \geq\inf_{z\in E}{\int\limits_{\mathbb{R}^{N}}}\left\vert u-z\right\vert
^{2}\mathrm{dx}.
\end{align*}
Moreover, it was showed in \cite{McV21} that $E$ is a closed cone (in the
$L^{2}$-norm). Therefore, we can find $v\in E$ such that $\inf_{z\in E}%
{\int\limits_{\mathbb{R}^{N}}}\left\vert u-z\right\vert ^{2}\mathrm{dx}%
={\int\limits_{\mathbb{R}^{N}}}\left\vert u-v\right\vert ^{2}\mathrm{dx}$.
Thus,%
\[
{\int\limits_{\mathbb{R}^{N}}}\left\vert u-v\right\vert ^{2}\mathrm{dx}%
\leq\delta_{1}\left(  u\right)  <1.
\]
Therefore, since ${\int\limits_{\mathbb{R}^{N}}}\left\vert u\right\vert
^{2}\mathrm{dx}=1$, we deduce that ${\int\limits_{\mathbb{R}^{N}}}\left\vert
v\right\vert ^{2}\mathrm{dx}\neq0$. Let $\lambda=\left(  {\int
\limits_{\mathbb{R}^{N}}}\left\vert v\right\vert ^{2}\mathrm{dx}\right)
^{-\frac{1}{2}}>0$. Then $w=\lambda v\in E$ and ${\int\limits_{\mathbb{R}^{N}%
}}\left\vert w\right\vert ^{2}\mathrm{dx}=1$. Also,%
\begin{align*}
{\int\limits_{\mathbb{R}^{N}}}\left\vert u-w\right\vert ^{2}\mathrm{dx}  &
={\int\limits_{\mathbb{R}^{N}}}\left\vert u-\lambda v\right\vert
^{2}\mathrm{dx}\\
&  =2-2\lambda{\int\limits_{\mathbb{R}^{N}}uv}\mathrm{dx}\text{.}%
\end{align*}
Note that since ${\int\limits_{\mathbb{R}^{N}}}\left\vert u-v\right\vert
^{2}\mathrm{dx}\leq\delta_{1}\left(  u\right)  <1$, we obtain%
\[
1-2{\int\limits_{\mathbb{R}^{N}}uv}\mathrm{dx}+\frac{1}{\lambda^{2}}\leq
\delta_{1}\left(  u\right)  <1
\]
and
\[
0<\frac{1}{2\lambda^{2}}\leq{\int\limits_{\mathbb{R}^{N}}uv}\mathrm{dx}.
\]
Now, to show that
\[
\delta_{1}\left(  u\right)  \geq\frac{1}{2}\inf_{u^{\ast}\in E}\left\{
{\int\limits_{\mathbb{R}^{N}}}\left\vert u-u^{\ast}\right\vert ^{2}%
\mathrm{dx}:{\int\limits_{\mathbb{R}^{N}}}\left\vert u\right\vert
^{2}\mathrm{dx}={\int\limits_{\mathbb{R}^{N}}}\left\vert u^{\ast}\right\vert
^{2}\mathrm{dx}\right\}  ,
\]
it's enough to prove that
\begin{align*}
\delta_{1}\left(  u\right)   &  \geq1-2{\int\limits_{\mathbb{R}^{N}}%
uv}\mathrm{dx}+\frac{1}{\lambda^{2}}\\
&  \geq1-\lambda{\int\limits_{\mathbb{R}^{N}}uv}\mathrm{dx}\\
&  =\frac{1}{2}{\int\limits_{\mathbb{R}^{N}}}\left\vert u-w\right\vert
^{2}\mathrm{dx}.
\end{align*}
That is%
\[
\left(  2-\lambda\right)  {\int\limits_{\mathbb{R}^{N}}uv}\mathrm{dx}\leq
\frac{1}{\lambda^{2}}.
\]
Indeed, note that by the H\"{o}lder inequality:
\[
\lambda{\int\limits_{\mathbb{R}^{N}}uv}\mathrm{dx}={\int\limits_{\mathbb{R}%
^{N}}uw}\mathrm{dx}\leq1\text{,}%
\]
we have
\[
\frac{1}{2\lambda^{2}}\leq{\int\limits_{\mathbb{R}^{N}}uv}\mathrm{dx}\leq
\frac{1}{\lambda}\text{.}%
\]
If $\lambda>2$, then $\left(  2-\lambda\right)  {\int\limits_{\mathbb{R}^{N}%
}uv}\mathrm{dx}<0<\frac{1}{\lambda^{2}}$. That is $1-\lambda{\int
\limits_{\mathbb{R}^{N}}uv}\mathrm{dx}\leq1-2{\int\limits_{\mathbb{R}^{N}}%
uv}\mathrm{dx}+\frac{1}{\lambda^{2}}\leq\delta_{1}\left(  u\right)  $. If
$0<\lambda\leq2$, then%
\begin{align*}
\left(  2-\lambda\right)  {\int\limits_{\mathbb{R}^{N}}uv}\mathrm{dx}  &
\leq\left(  2-\lambda\right)  \frac{1}{\lambda}\\
&  =\frac{2}{\lambda}-1\\
&  \leq\frac{1}{\lambda^{2}}.
\end{align*}

If
\[
\delta_{1}\left(  u\right)  =\left(  {\int\limits_{\mathbb{R}^{N}}}\left\vert
\nabla u\right\vert ^{2}\mathrm{dx}\right)  ^{1/2}\left(  {\int
\limits_{\mathbb{R}^{N}}}\left\vert x\right\vert ^{2}\left\vert u\right\vert
^{2}\mathrm{dx}\right)  ^{1/2}-\frac{N}{2}{\int\limits_{\mathbb{R}^{N}}%
}\left\vert u\right\vert ^{2}\mathrm{dx}\geq1\text{,}%
\]
then since ${\int\limits_{\mathbb{R}^{N}}}\left\vert u-u^{\ast}\right\vert
^{2}\mathrm{dx}+{\int\limits_{\mathbb{R}^{N}}}\left\vert u+u^{\ast}\right\vert
^{2}\mathrm{dx}=2\left(  {\int\limits_{\mathbb{R}^{N}}}\left\vert u\right\vert
^{2}\mathrm{dx}+{\int\limits_{\mathbb{R}^{N}}}\left\vert u^{\ast}\right\vert
^{2}\mathrm{dx}\right)  $, we get%
\begin{align*}
&  \inf_{u^{\ast}\in E}\left\{  {\int\limits_{\mathbb{R}^{N}}}\left\vert
u-u^{\ast}\right\vert ^{2}\mathrm{dx}:{\int\limits_{\mathbb{R}^{N}}}\left\vert
u\right\vert ^{2}\mathrm{dx}={\int\limits_{\mathbb{R}^{N}}}\left\vert u^{\ast
}\right\vert ^{2}\mathrm{dx}=1\right\}  \\
&  \leq\frac{1}{2}\inf_{u^{\ast}\in E}\left\{  {\int\limits_{\mathbb{R}^{N}}%
}\left\vert u-u^{\ast}\right\vert ^{2}\mathrm{dx}+{\int\limits_{\mathbb{R}%
^{N}}}\left\vert u+u^{\ast}\right\vert ^{2}\mathrm{dx}:{\int
\limits_{\mathbb{R}^{N}}}\left\vert u\right\vert ^{2}\mathrm{dx}%
={\int\limits_{\mathbb{R}^{N}}}\left\vert u^{\ast}\right\vert ^{2}%
\mathrm{dx}=1\right\}  \\
&  =2\text{.}%
\end{align*}
Therefore%
\[
\delta_{1}\left(  u\right)  \geq\frac{1}{2}\inf_{u^{\ast}\in E}\left\{
{\int\limits_{\mathbb{R}^{N}}}\left\vert u-u^{\ast}\right\vert ^{2}%
\mathrm{dx}:{\int\limits_{\mathbb{R}^{N}}}\left\vert u\right\vert
^{2}\mathrm{dx}={\int\limits_{\mathbb{R}^{N}}}\left\vert u^{\ast}\right\vert
^{2}\mathrm{dx}\right\}  .
\]
Now, let $u=x_{1}e^{-\frac{1}{2}\left\vert x\right\vert ^{2}}$. Then as in the
proof of Theorem \ref{T3.2}, we get
\begin{align*}
{\int\limits_{\mathbb{R}^{N}}}\left\vert \nabla u\right\vert ^{2}\mathrm{dx}
&  ={\int\limits_{\mathbb{R}^{N}}}\left\vert x\right\vert ^{2}\left\vert
u\right\vert ^{2}\mathrm{dx}\\
&  =\frac{N+2}{4}\pi^{\frac{N}{2}}%
\end{align*}
and%
\begin{align*}
&  \left(  {\int\limits_{\mathbb{R}^{N}}}\left\vert \nabla u\right\vert
^{2}\mathrm{dx}\right)  ^{\frac{1}{2}}\left(  {\int\limits_{\mathbb{R}^{N}}%
}\left\vert x\right\vert ^{2}\left\vert u\right\vert ^{2}\mathrm{dx}\right)
^{\frac{1}{2}}-\frac{N}{2}{\int\limits_{\mathbb{R}^{N}}}\left\vert
u\right\vert ^{2}\mathrm{dx}\\
&  =\frac{1}{2}\pi^{\frac{N}{2}}\text{.}%
\end{align*}
Also,
\[
{\int\limits_{\mathbb{R}^{N}}}\left\vert u-ce^{-\frac{1}{2\lambda^{2}%
}\left\vert x\right\vert ^{2}}\right\vert ^{2}\mathrm{dx}=\frac{1}{2}%
\pi^{\frac{N}{2}}+\left\vert c\right\vert ^{2}\lambda^{N}\pi^{\frac{N}{2}}.
\]
Note that
\[
{\int\limits_{\mathbb{R}^{N}}}\left\vert ce^{-\frac{1}{2\lambda^{2}}\left\vert
x\right\vert ^{2}}\right\vert ^{2}\mathrm{dx}=\left\vert c\right\vert
^{2}\lambda^{N}\pi^{\frac{N}{2}}\text{.}%
\]
Hence
\begin{align*}
&  \inf_{u^{\ast}\in E}\left\{  {\int\limits_{\mathbb{R}^{N}}}\left\vert
u-u^{\ast}\right\vert ^{2}\mathrm{dx}:{\int\limits_{\mathbb{R}^{N}}}\left\vert
u\right\vert ^{2}\mathrm{dx}={\int\limits_{\mathbb{R}^{N}}}\left\vert u^{\ast
}\right\vert ^{2}\mathrm{dx}\right\}  \\
&  =\inf_{c,\lambda}\left\{  \frac{1}{2}\pi^{\frac{N}{2}}+\left\vert
c\right\vert ^{2}\lambda^{N}\pi^{\frac{N}{2}}:\frac{1}{2}\pi^{\frac{N}{2}%
}=\left\vert c\right\vert ^{2}\lambda^{N}\pi^{\frac{N}{2}}\right\}  \\
&  =\pi^{\frac{N}{2}}.
\end{align*}
Therefore,%
\begin{align*}
&  \left(  {\int\limits_{\mathbb{R}^{N}}}\left\vert \nabla u\right\vert
^{2}\mathrm{dx}\right)  ^{\frac{1}{2}}\left(  {\int\limits_{\mathbb{R}^{N}}%
}\left\vert x\right\vert ^{2}\left\vert u\right\vert ^{2}\mathrm{dx}\right)
^{\frac{1}{2}}-\frac{N}{2}{\int\limits_{\mathbb{R}^{N}}}\left\vert
u\right\vert ^{2}\mathrm{dx}\\
&  =\frac{1}{2}\inf_{u^{\ast}\in E}\left\{  {\int\limits_{\mathbb{R}^{N}}%
}\left\vert u-u^{\ast}\right\vert ^{2}\mathrm{dx}:{\int\limits_{\mathbb{R}%
^{N}}}\left\vert u\right\vert ^{2}\mathrm{dx}={\int\limits_{\mathbb{R}^{N}}%
}\left\vert u^{\ast}\right\vert ^{2}\mathrm{dx}\right\}  .
\end{align*}

\end{proof}

Next, we note that it was showed by McCurdy and Venkatraman in \cite{McV21}
that there is no quantitative stability version of the scale invariant
Heisenberg Uncertainty Principle if we use the norm $\left\Vert \nabla\left(
u-u^{\ast}\right)  \right\Vert _{2}$ or $\left\Vert x\left(  u-u^{\ast
}\right)  \right\Vert _{2}$ on the right hand side. In particular, there is no
$C>0$ such that
\begin{align*}
&  \left(  {\int\limits_{\mathbb{R}^{N}}}\left\vert \nabla u\right\vert
^{2}\mathrm{dx}\right)  ^{1/2}\left(  {\int\limits_{\mathbb{R}^{N}}}\left\vert
x\right\vert ^{2}\left\vert u\right\vert ^{2}\mathrm{dx}\right)  ^{1/2}%
-\frac{N}{2}{\int\limits_{\mathbb{R}^{N}}}\left\vert u\right\vert
^{2}\mathrm{dx}\\
&  \geq C\inf_{u^{\ast}\in E}{\int\limits_{\mathbb{R}^{N}}}\left\vert
\nabla\left(  u-u^{\ast}\right)  \right\vert ^{2}\mathrm{dx}+{\int
\limits_{\mathbb{R}^{N}}}\left\vert x\right\vert ^{2}\left\vert u-u^{\ast
}\right\vert ^{2}\mathrm{dx}+{\int\limits_{\mathbb{R}^{N}}}\left\vert
u-u^{\ast}\right\vert ^{2}\mathrm{dx}.
\end{align*}
We now show that this is not the case for the scale non-invariant Heisenberg
Uncertainty Principle, namely, Theorem \ref{E4}:

\begin{theorem}
\label{T3.4}For all $u\in X$, then%
\begin{align*}
&  {\int\limits_{\mathbb{R}^{N}}}\left\vert \nabla u\right\vert ^{2}%
\mathrm{dx}+{\int\limits_{\mathbb{R}^{N}}}\left\vert x\right\vert
^{2}\left\vert u\right\vert ^{2}\mathrm{dx}-N{\int\limits_{\mathbb{R}^{N}}%
}\left\vert u\right\vert ^{2}\mathrm{dx}\\
&  ={\int\limits_{\mathbb{R}^{N}}}\left\vert \nabla\left(  ue^{\frac{1}%
{2}\left\vert x\right\vert ^{2}}\right)  \right\vert ^{2}e^{-\left\vert
x\right\vert ^{2}}\mathrm{dx}\\
&  \geq\frac{2}{N+3}\inf_{c\in\mathbb{R}}\left(  {\int\limits_{\mathbb{R}^{N}%
}}\left\vert \nabla\left(  u-ce^{-\frac{1}{2}\left\vert x\right\vert ^{2}%
}\right)  \right\vert ^{2}\mathrm{dx}+{\int\limits_{\mathbb{R}^{N}}}\left\vert
x\right\vert ^{2}\left\vert u-ce^{-\frac{1}{2}\left\vert x\right\vert ^{2}%
}\right\vert ^{2}\mathrm{dx}+{\int\limits_{\mathbb{R}^{N}}}\left\vert
u-ce^{-\frac{1}{2}\left\vert x\right\vert ^{2}}\right\vert ^{2}\mathrm{dx}%
\right)  .
\end{align*}
The inequality is sharp and the equality can be attained by nontrivial functions.
\end{theorem}

\begin{proof}
Indeed, let $u=ve^{-\frac{1}{2}\left\vert x\right\vert ^{2}}$, then
\begin{align*}
&  {\int\limits_{\mathbb{R}^{N}}}\left\vert \nabla\left(  u-ce^{-\frac{1}%
{2}\left\vert x\right\vert ^{2}}\right)  \right\vert ^{2}\mathrm{dx}\\
&  ={\int\limits_{\mathbb{R}^{N}}}\left\vert \nabla\left[  \left(  v-c\right)
e^{-\frac{1}{2}\left\vert x\right\vert ^{2}}\right]  \right\vert
^{2}\mathrm{dx}\\
&  ={\int\limits_{\mathbb{R}^{N}}}\left\vert \nabla v\right\vert
^{2}e^{-\left\vert x\right\vert ^{2}}\mathrm{dx}-2{\int\limits_{\mathbb{R}%
^{N}}}\left(  v-c\right)  x\cdot\nabla\left(  v-c\right)  e^{-\left\vert
x\right\vert ^{2}}\mathrm{dx}+{\int\limits_{\mathbb{R}^{N}}}\left\vert
x\right\vert ^{2}\left\vert v-c\right\vert ^{2}e^{-\left\vert x\right\vert
^{2}}\mathrm{dx}\\
&  ={\int\limits_{\mathbb{R}^{N}}}\left\vert \nabla v\right\vert
^{2}e^{-\left\vert x\right\vert ^{2}}\mathrm{dx}+N{\int\limits_{\mathbb{R}%
^{N}}}\left\vert v-c\right\vert ^{2}e^{-\left\vert x\right\vert ^{2}%
}\mathrm{dx}-{\int\limits_{\mathbb{R}^{N}}}\left\vert x\right\vert
^{2}\left\vert v-c\right\vert ^{2}e^{-\left\vert x\right\vert ^{2}}%
\mathrm{dx}.
\end{align*}
Therefore%
\begin{align*}
&  {\int\limits_{\mathbb{R}^{N}}}\left\vert \nabla\left(  u-ce^{-\frac{1}%
{2}\left\vert x\right\vert ^{2}}\right)  \right\vert ^{2}\mathrm{dx}%
+{\int\limits_{\mathbb{R}^{N}}}\left\vert x\right\vert ^{2}\left\vert
u-ce^{-\frac{1}{2}\left\vert x\right\vert ^{2}}\right\vert ^{2}\mathrm{dx}%
+{\int\limits_{\mathbb{R}^{N}}}\left\vert u-ce^{-\frac{1}{2}\left\vert
x\right\vert ^{2}}\right\vert ^{2}\mathrm{dx}\\
&  ={\int\limits_{\mathbb{R}^{N}}}\left\vert \nabla\left[  \left(  v-c\right)
e^{-\frac{1}{2}\left\vert x\right\vert ^{2}}\right]  \right\vert
^{2}\mathrm{dx}+{\int\limits_{\mathbb{R}^{N}}}\left\vert x\right\vert
^{2}\left\vert v-c\right\vert ^{2}e^{-\left\vert x\right\vert ^{2}}%
\mathrm{dx}+{\int\limits_{\mathbb{R}^{N}}}\left\vert v-c\right\vert
^{2}e^{-\left\vert x\right\vert ^{2}}\mathrm{dx}\\
&  ={\int\limits_{\mathbb{R}^{N}}}\left\vert \nabla v\right\vert
^{2}e^{-\left\vert x\right\vert ^{2}}\mathrm{dx}+\left(  N+1\right)
{\int\limits_{\mathbb{R}^{N}}}\left\vert v-c\right\vert ^{2}e^{-\left\vert
x\right\vert ^{2}}\mathrm{dx}.
\end{align*}
Noting that by the Poincar\'{e} inequality, we get
\[
{\int\limits_{\mathbb{R}^{N}}}\left\vert \nabla v\right\vert ^{2}%
e^{-\left\vert x\right\vert ^{2}}\mathrm{dx}\geq2\inf_{c}{\int
\limits_{\mathbb{R}^{N}}}\left\vert v-c\right\vert ^{2}e^{-\left\vert
x\right\vert ^{2}}\mathrm{dx}.
\]
Hence
\begin{align*}
\frac{2}{N+3}  &  \inf_{c\in\mathbb{R}}{\int\limits_{\mathbb{R}^{N}}%
}\left\vert \nabla\left(  u-ce^{-\frac{1}{2}\left\vert x\right\vert ^{2}%
}\right)  \right\vert ^{2}\mathrm{dx}+{\int\limits_{\mathbb{R}^{N}}}\left\vert
x\right\vert ^{2}\left\vert u-ce^{-\frac{1}{2}\left\vert x\right\vert ^{2}%
}\right\vert ^{2}\mathrm{dx}+{\int\limits_{\mathbb{R}^{N}}}\left\vert
u-ce^{-\frac{1}{2}\left\vert x\right\vert ^{2}}\right\vert ^{2}\mathrm{dx}\\
&  \leq\frac{2}{N+3}{\int\limits_{\mathbb{R}^{N}}}\left\vert \nabla
v\right\vert ^{2}e^{-\left\vert x\right\vert ^{2}}\mathrm{dx}+\frac{2}%
{N+3}\left(  N+1\right)  \inf_{c\in\mathbb{R}}{\int\limits_{\mathbb{R}^{N}}%
}\left\vert v-c\right\vert ^{2}e^{-\left\vert x\right\vert ^{2}}\mathrm{dx}\\
&  \leq{\int\limits_{\mathbb{R}^{N}}}\left\vert \nabla v\right\vert
^{2}e^{-\left\vert x\right\vert ^{2}}\mathrm{dx}\\
&  ={\int\limits_{\mathbb{R}^{N}}}\left\vert \nabla\left(  ue^{\frac{1}%
{2}\left\vert x\right\vert ^{2}}\right)  \right\vert ^{2}e^{-\left\vert
x\right\vert ^{2}}\mathrm{dx}\\
&  ={\int\limits_{\mathbb{R}^{N}}}\left\vert \nabla u\right\vert
^{2}\mathrm{dx}+{\int\limits_{\mathbb{R}^{N}}}\left\vert x\right\vert
^{2}\left\vert u\right\vert ^{2}\mathrm{dx}-N{\int\limits_{\mathbb{R}^{N}}%
}\left\vert u\right\vert ^{2}\mathrm{dx}.
\end{align*}

Let $u=x_{1}e^{-\frac{1}{2}\left\vert x\right\vert ^{2}}$ and $v=x_{1}$. Then%
\begin{align*}
&  {\int\limits_{\mathbb{R}^{N}}}\left\vert \nabla u\right\vert ^{2}%
\mathrm{dx}+{\int\limits_{\mathbb{R}^{N}}}\left\vert x\right\vert
^{2}\left\vert u\right\vert ^{2}\mathrm{dx}-N{\int\limits_{\mathbb{R}^{N}}%
}\left\vert u\right\vert ^{2}\mathrm{dx}\\
&  =\pi^{\frac{N}{2}}%
\end{align*}
and
\begin{align*}
&  {\int\limits_{\mathbb{R}^{N}}}\left\vert \nabla\left(  u-ce^{-\frac{1}%
{2}\left\vert x\right\vert ^{2}}\right)  \right\vert ^{2}\mathrm{dx}%
+{\int\limits_{\mathbb{R}^{N}}}\left\vert x\right\vert ^{2}\left\vert
u-ce^{-\frac{1}{2}\left\vert x\right\vert ^{2}}\right\vert ^{2}\mathrm{dx}%
+{\int\limits_{\mathbb{R}^{N}}}\left\vert u-ce^{-\frac{1}{2}\left\vert
x\right\vert ^{2}}\right\vert ^{2}\mathrm{dx}\\
&  =\left\vert 1+\left(  N+1\right)  \left(  \frac{1}{2}+\left\vert
c\right\vert ^{2}\right)  \right\vert \pi^{\frac{N}{2}}.
\end{align*}
Therefore
\begin{align*}
&  \inf_{c\in\mathbb{R}}{\int\limits_{\mathbb{R}^{N}}}\left\vert \nabla\left(
u-ce^{-\frac{1}{2}\left\vert x\right\vert ^{2}}\right)  \right\vert
^{2}\mathrm{dx}+{\int\limits_{\mathbb{R}^{N}}}\left\vert x\right\vert
^{2}\left\vert u-ce^{-\frac{1}{2}\left\vert x\right\vert ^{2}}\right\vert
^{2}\mathrm{dx}+{\int\limits_{\mathbb{R}^{N}}}\left\vert u-ce^{-\frac{1}%
{2}\left\vert x\right\vert ^{2}}\right\vert ^{2}\mathrm{dx}\\
&  =\frac{N+3}{2}\pi^{\frac{N}{2}}\\
&  =\frac{N+3}{2}\left(  {\int\limits_{\mathbb{R}^{N}}}\left\vert \nabla
u\right\vert ^{2}\mathrm{dx}+{\int\limits_{\mathbb{R}^{N}}}\left\vert
x\right\vert ^{2}\left\vert u\right\vert ^{2}\mathrm{dx}-N{\int
\limits_{\mathbb{R}^{N}}}\left\vert u\right\vert ^{2}\mathrm{dx}\right)  .
\end{align*}

\end{proof}

\subsection{The stability of the Caffarelli-Kohn-Nirenberg inequalities}

Let $X_{a,b}$ be the completion of $C_{0}^{\infty}\left(  \mathbb{R}%
^{N}\right)  $ under the norm $\left(  {\int\limits_{\mathbb{R}^{N}}}%
\frac{\left\vert \nabla u\right\vert ^{2}}{\left\vert x\right\vert ^{2b}%
}\mathrm{dx}+{\int\limits_{\mathbb{R}^{N}}}\frac{\left\vert u\right\vert ^{2}%
}{\left\vert x\right\vert ^{2a}}\mathrm{dx}\right)  ^{\frac{1}{2}}$. We now
recall that by Corollary \ref{c4}, we have the following scale non-invariant
Caffarelli-Kohn-Nirenberg inequalities with remainders:

\begin{lemma}
\label{l2}Let $b+1-a>0$ and $b<\frac{N-2}{2}$. Then for $u\in X_{a,b}:$
\begin{align}
&  {\int\limits_{\mathbb{R}^{N}}}\frac{\left\vert \nabla u\right\vert ^{2}%
}{\left\vert x\right\vert ^{2b}}\mathrm{dx}+{\int\limits_{\mathbb{R}^{N}}%
}\frac{\left\vert u\right\vert ^{2}}{\left\vert x\right\vert ^{2a}}%
\mathrm{dx}-\left(  N-a-b-1\right)  {\int\limits_{\mathbb{R}^{N}}}%
\frac{\left\vert u\right\vert ^{2}}{\left\vert x\right\vert ^{a+b+1}%
}\mathrm{dx}\nonumber\\
&  ={\int\limits_{\mathbb{R}^{N}}}\frac{1}{\left\vert x\right\vert ^{2b}%
}\left\vert \nabla\left(  ue^{\frac{1}{b+1-a}\left\vert x\right\vert ^{b+1-a}%
}\right)  \right\vert ^{2}e^{-\frac{2}{b+1-a}\left\vert x\right\vert ^{b+1-a}%
}\mathrm{dx}. \label{CKN1}%
\end{align}

\end{lemma}

We also state here the $L^{2}$-Caffarelli-Kohn-Nirenberg inequalities

\begin{lemma}
\label{l3}Let $b+1-a>0$ and $b<\frac{N-2}{2}$. Then for $u\in X_{a,b}:$%
\begin{equation}
\left(  {\int\limits_{\mathbb{R}^{N}}}\frac{\left\vert \nabla u\right\vert
^{2}}{\left\vert x\right\vert ^{2b}}\mathrm{dx}\right)  \left(  {\int
\limits_{\mathbb{R}^{N}}}\frac{\left\vert u\right\vert ^{2}}{\left\vert
x\right\vert ^{2a}}\mathrm{dx}\right)  \geq\left\vert \frac{N-a-b-1}%
{2}\right\vert ^{2}\left(  {\int\limits_{\mathbb{R}^{N}}}\frac{\left\vert
u\right\vert ^{2}}{\left\vert x\right\vert ^{a+b+1}}\mathrm{dx}\right)  ^{2}.
\label{CKN2}%
\end{equation}
The equality happens if and only if $\displaystyle u(x)=\alpha\exp
(-\frac{\beta}{b+1-a}|x|^{b+1-a})$ with $\alpha\in\mathbb{R}$, $\beta>0.$
\end{lemma}

Therefore we now can define the Caffarelli-Kohn-Nirenberg deficit%
\[
\delta_{1,a,b}\left(  u\right)  :=\left(  {\int\limits_{\mathbb{R}^{N}}}%
\frac{|u|^{2}}{|x|^{2a}}\mathrm{dx}\right)  ^{\frac{1}{2}}\left(
{\int\limits_{\mathbb{R}^{N}}}\frac{|\nabla u|^{2}}{|x|^{2b}}\mathrm{dx}%
\right)  ^{\frac{1}{2}}-\left\vert \frac{N-a-b-1}{2}\right\vert \left(
{\int\limits_{\mathbb{R}^{N}}}\frac{|u|^{2}}{|x|^{a+b+1}}\mathrm{dx}\right)
.
\]

We use the following Poincar\'{e} inequality in \cite[Chapter 4]{BGL14}:

\begin{lemma}
\label{l4}Let $\mu$ be a log-concave probability measure on $\mathbb{R}^{N}$.
Then $\mu$ satisfies a Poincar\'{e} inequality:%
\[
{\int\limits_{\mathbb{R}^{N}}}\left\vert \nabla v\right\vert ^{2}d\mu
\gtrsim{\int\limits_{\mathbb{R}^{N}}}\left\vert v-{\int\limits_{\mathbb{R}%
^{N}}v}d\mu\right\vert ^{2}d\mu.
\]

\end{lemma}

We note that when $a\leq0$, then $\mu\left(  x\right)  =\frac{e^{-\frac
{2}{1-a}\left\vert x\right\vert ^{1-a}}}{{\int\limits_{\mathbb{R}^{N}}%
}e^{-\frac{2}{1-a}\left\vert x\right\vert ^{1-a}}\mathrm{dx}}$ is a
log-concave probability measure on $\mathbb{R}^{N}$. Therefore, there exists
$C\left(  N,a\right)  >0$ such that the following Poincar\'{e} inequality
holds%
\[
{\int\limits_{\mathbb{R}^{N}}}\left\vert \nabla v\right\vert ^{2}e^{-\frac
{2}{1-a}\left\vert x\right\vert ^{1-a}}\mathrm{dx}\geq C\left(  N,a\right)
\inf_{c\in\mathbb{R}}{\int\limits_{\mathbb{R}^{N}}}\left\vert v-c\right\vert
^{2}e^{-\frac{2}{1-a}\left\vert x\right\vert ^{1-a}}\mathrm{dx}.
\]
Therefore, we obtain the following result:

\begin{theorem}
\label{T3.5}Let $a\leq0$. Then we have that for $u\in X_{a,0}:$%
\begin{align*}
&  {\int\limits_{\mathbb{R}^{N}}}\left\vert \nabla u\right\vert ^{2}%
\mathrm{dx}+{\int\limits_{\mathbb{R}^{N}}}\frac{\left\vert u\right\vert ^{2}%
}{\left\vert x\right\vert ^{2a}}\mathrm{dx}-\left(  N-a-1\right)
{\int\limits_{\mathbb{R}^{N}}}\frac{\left\vert u\right\vert ^{2}}{\left\vert
x\right\vert ^{a+1}}\mathrm{dx}\\
&  \geq C\left(  N,a\right)  \inf_{c\in\mathbb{R}}{\int\limits_{\mathbb{R}%
^{N}}}\left\vert u-ce^{-\frac{1}{1-a}\left\vert x\right\vert ^{1-a}%
}\right\vert ^{2}\mathrm{dx}.
\end{align*}

\end{theorem}

\begin{proof}
For $u\in X_{a,0}:$%
\begin{align*}
&  {\int\limits_{\mathbb{R}^{N}}}\left\vert \nabla u\right\vert ^{2}%
\mathrm{dx}+{\int\limits_{\mathbb{R}^{N}}}\frac{\left\vert u\right\vert ^{2}%
}{\left\vert x\right\vert ^{2a}}\mathrm{dx}-\left(  N-a-1\right)
{\int\limits_{\mathbb{R}^{N}}}\frac{\left\vert u\right\vert ^{2}}{\left\vert
x\right\vert ^{a+1}}\mathrm{dx}\\
&  ={\int\limits_{\mathbb{R}^{N}}}\left\vert \nabla\left(  ue^{\frac{1}%
{1-a}\left\vert x\right\vert ^{1-a}}\right)  \right\vert ^{2}e^{-\frac{2}%
{1-a}\left\vert x\right\vert ^{1-a}}\mathrm{dx}\\
&  \geq C\left(  N,a\right)  \inf_{c\in\mathbb{R}}{\int\limits_{\mathbb{R}%
^{N}}}\left\vert ue^{\frac{1}{1-a}\left\vert x\right\vert ^{1-a}}-c\right\vert
^{2}e^{-\frac{2}{1-a}\left\vert x\right\vert ^{1-a}}\mathrm{dx}\\
&  =C\left(  N,a\right)  \inf_{c\in\mathbb{R}}{\int\limits_{\mathbb{R}^{N}}%
}\left\vert u-ce^{-\frac{1}{1-a}\left\vert x\right\vert ^{1-a}}\right\vert
^{2}\mathrm{dx}.
\end{align*}

\end{proof}

In order to study the stability results for the Caffarelli-Kohn-Nirenberg
inequalities, we will now establish a weighted Poincar\'{e} inequality for the
log-concave probability measure. More precisely, we will prove the following
weighted Poincar\'{e} inequality:

\begin{lemma}
\label{key2}For some $\delta>0$, $N-2>\mu\geq0$ and $\alpha\geq\frac{N-2-\mu
}{N-2}:$%
\[
{\int\limits_{\mathbb{R}^{N}}}\frac{\left\vert \nabla v\left(  y\right)
\right\vert ^{2}}{\left\vert y\right\vert ^{\mu}}e^{-\delta\left\vert
y\right\vert ^{\alpha}}dy\geq C\left(  N,\alpha,\delta,\mu\right)  \inf
_{c}{\int\limits_{\mathbb{R}^{N}}}\frac{\left\vert v\left(  y\right)
-c\right\vert ^{2}}{\left\vert y\right\vert ^{\frac{N\mu}{N-2}}}%
e^{-\delta\left\vert y\right\vert ^{\alpha}}dy.
\]

\end{lemma}

\begin{proof}
Let $\overline{v}\left(  x\right)  =\left(  \frac{1}{\lambda}\right)
^{\frac{1}{2}}v\left(  \left\vert x\right\vert ^{\lambda-1}x\right)  $. Note
that for the change of variable $x\rightarrow\left\vert x\right\vert
^{\lambda-1}x$, the Jacobian is $\lambda\left\vert x\right\vert ^{N\left(
\lambda-1\right)  }$. Now, for $\lambda\geq1$, we can show that%
\[
\left\vert \nabla\overline{v}\left(  x\right)  \right\vert \leq\lambda
^{\frac{1}{2}}\left\vert x\right\vert ^{\lambda-1}\left\vert \nabla v\left(
\left\vert x\right\vert ^{\lambda-1}x\right)  \right\vert .
\]
See \cite{LL17}. Therefore (set $y=\left\vert x\right\vert ^{\lambda-1}x$)
\begin{align*}
{\int\limits_{\mathbb{R}^{N}}}\frac{\left\vert \nabla v\left(  y\right)
\right\vert ^{2}}{\left\vert y\right\vert ^{\mu}}e^{-\delta\left\vert
y\right\vert ^{\alpha}}dy  &  ={\int\limits_{\mathbb{R}^{N}}}\frac{\left\vert
\nabla v\left(  \left\vert x\right\vert ^{\lambda-1}x\right)  \right\vert
^{2}}{\left\vert x\right\vert ^{\lambda\mu}}e^{-\delta\left\vert x\right\vert
^{\lambda\alpha}}d\left\vert x\right\vert ^{N\left(  \lambda-1\right)
}\mathrm{dx}\\
&  \geq{\int\limits_{\mathbb{R}^{N}}}\frac{\left\vert \nabla\overline
{v}\left(  x\right)  \right\vert ^{2}}{\lambda\left\vert x\right\vert
^{2\left(  \lambda-1\right)  +\lambda\mu-N\left(  \lambda-1\right)  }%
}e^{-\delta\left\vert x\right\vert ^{\lambda\alpha}}\lambda\mathrm{dx}\\
&  ={\int\limits_{\mathbb{R}^{N}}}\frac{\left\vert \nabla\overline{v}\left(
x\right)  \right\vert ^{2}}{\left\vert x\right\vert ^{\lambda\left(
2+\mu-N\right)  +N-2}}e^{-\delta\left\vert x\right\vert ^{\lambda\alpha}%
}\mathrm{dx}.
\end{align*}
So if we choose $\lambda=\frac{N-2}{N-2-\mu}\geq1$, then by noting that the
measure $e^{-\delta\left\vert x\right\vert ^{\frac{N-2}{N-2-\mu}\alpha}}$ is
log-concave, we obtain
\begin{align*}
{\int\limits_{\mathbb{R}^{N}}}\frac{\left\vert \nabla v\left(  y\right)
\right\vert ^{2}}{\left\vert y\right\vert ^{\mu}}e^{-\delta\left\vert
y\right\vert ^{\alpha}}dy  &  \geq{\int\limits_{\mathbb{R}^{N}}}\left\vert
\nabla\overline{v}\left(  x\right)  \right\vert ^{2}e^{-\delta\left\vert
x\right\vert ^{\frac{N-2}{N-2-\mu}\alpha}}\mathrm{dx}\\
&  \geq C\left(  N,\alpha,\delta,\mu\right)  \inf_{c}{\int\limits_{\mathbb{R}%
^{N}}}\left\vert \overline{v}\left(  x\right)  -c\right\vert ^{2}%
e^{-\delta\left\vert x\right\vert ^{\lambda\alpha}}\mathrm{dx}\\
&  =C\left(  N,\alpha,\delta,\mu\right)  \inf_{c}{\int\limits_{\mathbb{R}^{N}%
}}\frac{\left\vert v\left(  \left\vert x\right\vert ^{\lambda-1}x\right)
-c\right\vert ^{2}}{\left\vert x\right\vert ^{N\left(  \lambda-1\right)  }%
}e^{-\delta\left\vert x\right\vert ^{\lambda\alpha}}\left\vert x\right\vert
^{N\left(  \lambda-1\right)  }\mathrm{dx}\\
&  =C\left(  N,\alpha,\delta,\mu\right)  \inf_{c}{\int\limits_{\mathbb{R}^{N}%
}}\frac{\left\vert v\left(  y\right)  -c\right\vert ^{2}}{\left\vert
y\right\vert ^{\frac{N\mu}{N-2}}}e^{-\delta\left\vert y\right\vert ^{\alpha}%
}dy.
\end{align*}

\end{proof}

Let $E_{a,b}=\left\{  v(x)=\alpha\exp(-\frac{\beta}{b+1-a}|x|^{b+1-a})\text{
with }\alpha\in\mathbb{R},~\beta>0\right\}  $ be the set of optimizers for the
scale invariant Caffarelli-Kohn-Nirenberg inequalities (\ref{CKN2}) and
$d_{1,a,b}\left(  u,E_{a,b}\right)  :=\inf_{v\in E_{a,b}}\left(
{\int\limits_{\mathbb{R}^{N}}}\frac{\left\vert u-v\right\vert ^{2}}{\left\vert
x\right\vert ^{a+b+1}}\mathrm{dx}\right)  ^{\frac{1}{2}}$.

\begin{lemma}
\label{l5}Let $0\leq b<\frac{N-2}{2}$ and $a\leq\frac{Nb}{N-2}$. Then
$E_{a,b}$ is closed under the norm $\left(  {\int\limits_{\mathbb{R}^{N}}%
}\frac{\left\vert \cdot\right\vert ^{2}}{\left\vert x\right\vert ^{a+b+1}%
}\mathrm{dx}\right)  ^{\frac{1}{2}}$ and $d_{1,a,b}\left(  u,E_{a,b}\right)
:=\inf_{v\in E_{a,b}}\left(  {\int\limits_{\mathbb{R}^{N}}}\frac{\left\vert
u-v\right\vert ^{2}}{\left\vert x\right\vert ^{a+b+1}}\mathrm{dx}\right)
^{\frac{1}{2}}$ is the distance from $u$ to $E_{a,b}$. Moreover, for each
$u\in X_{a,b}$, there exists $u^{\ast}\in E_{a,b}$ such that $d_{1,a,b}\left(
u,E_{a,b}\right)  =\left(  {\int\limits_{\mathbb{R}^{N}}}\frac{\left\vert
u-u^{\ast}\right\vert ^{2}}{\left\vert x\right\vert ^{a+b+1}}\mathrm{dx}%
\right)  ^{\frac{1}{2}}$.
\end{lemma}

\begin{proof}
Let $u$ be such that ${\int\limits_{\mathbb{R}^{N}}}\frac{\left\vert
u\right\vert ^{2}}{\left\vert x\right\vert ^{a+b+1}}\mathrm{dx}=1$ and
$\inf_{v\in E}{\int\limits_{\mathbb{R}^{N}}}\frac{\left\vert u-v\right\vert
^{2}}{\left\vert x\right\vert ^{a+b+1}}\mathrm{dx}=0$. That is, there exists a
sequence $v_{j}\left(  x\right)  =\alpha_{j}\exp(-\frac{\beta_{j}}%
{b+1-a}|x|^{b+1-a})$ such that ${\int\limits_{\mathbb{R}^{N}}}\frac{\left\vert
v_{j}-u\right\vert ^{2}}{\left\vert x\right\vert ^{a+b+1}}\mathrm{dx}%
\rightarrow0$. Therefore, by taking subsequences, $v_{j}\rightarrow u$
pointwise a.e. in $\mathbb{R}^{N}$. Also, since ${\int\limits_{\mathbb{R}^{N}%
}}\frac{\left\vert u\right\vert ^{2}}{\left\vert x\right\vert ^{a+b+1}%
}\mathrm{dx}=1$ and ${\int\limits_{\mathbb{R}^{N}}}\frac{\left\vert
v_{j}-u\right\vert ^{2}}{\left\vert x\right\vert ^{a+b+1}}\mathrm{dx}%
\rightarrow0$, we get that when $j$ is large enough, ${\int\limits_{\mathbb{R}%
^{N}}}\frac{\left\vert v_{j}\right\vert ^{2}}{\left\vert x\right\vert
^{a+b+1}}\mathrm{dx}\in\left(  \frac{1}{2},\frac{3}{2}\right)  $. Then by
using the fact that
\[
\int\limits_{0}^{\infty}r^{x-1}e^{-\frac{\beta}{y}r^{y}}dr=\left(  \frac
{y}{\beta}\right)  ^{\frac{x}{y}}\frac{1}{y}\Gamma\left(  \frac{x}{y}\right)
\]
we get%
\[
{\int\limits_{\mathbb{R}^{N}}}\frac{\left\vert v_{j}\right\vert ^{2}%
}{\left\vert x\right\vert ^{a+b+1}}\mathrm{dx}=\left\vert \alpha
_{j}\right\vert ^{2}\left\vert S^{N-1}\right\vert \left(  \frac{b+1-a}%
{2\beta_{j}}\right)  ^{\frac{N-a-b-1}{b+1-a}}\frac{\Gamma\left(
\frac{N-a-b-1}{b+1-a}\right)  }{b+1-a}%
\]%
\[
{\int\limits_{\mathbb{R}^{N}}}\frac{|v_{j}|^{2}}{|x|^{2a}}\mathrm{dx}%
=\left\vert \alpha_{j}\right\vert ^{2}\left\vert S^{N-1}\right\vert \left(
\frac{b+1-a}{2\beta_{j}}\right)  ^{\frac{N-2a}{b+1-a}}\frac{\Gamma\left(
\frac{N-2a}{b+1-a}\right)  }{b+1-a}%
\]
and
\[
{\int\limits_{\mathbb{R}^{N}}}\frac{\left\vert \nabla v_{j}\right\vert ^{2}%
}{\left\vert x\right\vert ^{2b}}\mathrm{dx}=\left\vert \alpha_{j}\right\vert
^{2}\left\vert \beta_{j}\right\vert ^{2}\left\vert S^{N-1}\right\vert \left(
\frac{b+1-a}{2\beta_{j}}\right)  ^{\frac{N-2a}{b+1-a}}\frac{\Gamma\left(
\frac{N-2a}{b+1-a}\right)  }{b+1-a}.
\]
Therefore, $\frac{\left\vert \alpha_{j}\right\vert ^{2}}{\left\vert \beta
_{j}\right\vert ^{\frac{N-a-b-1}{b+1-a}}}=O\left(  1\right)  $. Hence if
$\lim\inf_{j\rightarrow\infty}\beta_{j}=0$, then
\begin{align*}
{\int\limits_{\mathbb{R}^{N}}}\frac{\left\vert \nabla v_{j}\right\vert ^{2}%
}{\left\vert x\right\vert ^{2b}}\mathrm{dx}  &  =O\left(  \frac{\left\vert
\alpha_{j}\right\vert ^{2}}{\left\vert \beta_{j}\right\vert ^{\frac
{N-2a}{b+1-a}-2}}\right) \\
&  =O\left(  \frac{\left\vert \alpha_{j}\right\vert ^{2}}{\left\vert \beta
_{j}\right\vert ^{\frac{N-a-b-1}{b+1-a}}}\right)  \left\vert \beta
_{j}\right\vert \\
&  \rightarrow0\text{.}%
\end{align*}
Now, fix $0<r<R<\infty$, then the sequence $v_{j}$ is bounded in
$W^{1,2}\left(  B_{R}\setminus B_{r}\right)  $. Hence, since $v_{j}\rightarrow
u$ pointwise a.e. in $\mathbb{R}^{N}$, we have that $v_{j}\rightharpoonup u$
weakly in $W^{1,2}\left(  B_{R}\setminus B_{r}\right)  $. Therefore,
\[
{\int\limits_{B_{R}\setminus B_{r}}}\frac{\left\vert \nabla u\right\vert ^{2}%
}{\left\vert x\right\vert ^{2b}}\mathrm{dx}\leq\lim\inf{\int\limits_{B_{R}%
\setminus B_{r}}}\frac{\left\vert \nabla v_{j}\right\vert ^{2}}{\left\vert
x\right\vert ^{2b}}\mathrm{dx}\leq\lim\inf{\int\limits_{\mathbb{R}^{N}}}%
\frac{\left\vert \nabla v_{j}\right\vert ^{2}}{\left\vert x\right\vert ^{2b}%
}\mathrm{dx}=0\text{.}%
\]
Now, sending $r\longrightarrow0$ and $R\rightarrow\infty$ gives us that%
\[
{\int\limits_{\mathbb{R}^{N}}}\frac{\left\vert \nabla u\right\vert ^{2}%
}{\left\vert x\right\vert ^{2b}}\mathrm{dx}=0
\]
which is impossible.

If $\lim\sup_{j\rightarrow\infty}\beta_{j}=\infty$, then similarly
\[
{\int\limits_{\mathbb{R}^{N}}}\frac{|v_{j}|^{2}}{|x|^{2a}}\mathrm{dx}%
\rightarrow0
\]
and so by Fatou's lemma
\[
{\int\limits_{\mathbb{R}^{N}}}\frac{|u|^{2}}{|x|^{2a}}\mathrm{dx}\leq\lim
\inf{\int\limits_{\mathbb{R}^{N}}}\frac{|v_{j}|^{2}}{|x|^{2a}}\mathrm{dx}=0
\]
which is impossible. Therefore, by taking subsequences, $\lim_{j\rightarrow
\infty}\beta_{j}=\beta\in\left(  0,\infty\right)  $. This implies that
$\lim_{j\rightarrow\infty}\alpha_{j}=\alpha\in\left(  0,\infty\right)  $.
Therefore, $v_{j}\left(  x\right)  =\alpha_{j}\exp(-\frac{\beta_{j}}%
{b+1-a}|x|^{b+1-a})\rightarrow\alpha\exp(-\frac{\beta}{b+1-a}|x|^{b+1-a})$
pointwise a.e. in $\mathbb{R}^{N}$. Hence, $u\left(  x\right)  =\alpha
\exp(-\frac{\beta}{b+1-a}|x|^{b+1-a})$ a.e. in $\mathbb{R}^{N}$. In other
words, $E_{a,b}$ is closed under the norm $\left(  {\int\limits_{\mathbb{R}%
^{N}}}\frac{\left\vert \cdot\right\vert ^{2}}{\left\vert x\right\vert
^{a+b+1}}\mathrm{dx}\right)  ^{\frac{1}{2}}$. Therefore, $d_{1,a,b}\left(
u,E_{a,b}\right)  $ is the distance from $u$ to $E_{a,b}$. That is,
$d_{1,a,b}\left(  u,E_{a,b}\right)  >0$ when $u\notin E_{a,b}$. Moreover, for
each $u\in X_{a,b}$, there exists $u^{\ast}\in E_{a,b}$ such that
$d_{1,a,b}\left(  u,E_{a,b}\right)  =\left(  {\int\limits_{\mathbb{R}^{N}}%
}\frac{\left\vert u-u^{\ast}\right\vert ^{2}}{\left\vert x\right\vert
^{a+b+1}}\mathrm{dx}\right)  ^{\frac{1}{2}}$.
\end{proof}

Now, we are ready to set up some results about the stability for the
Caffarelli-Kohn-Nirenberg inequalities, namely, Theorem \ref{E1}:

\begin{theorem}
\label{T3.6}Let $0\leq b<\frac{N-2}{2}$ and $a\leq\frac{Nb}{N-2}$. There
exists a universal constant $C_{1}\left(  N,a,b\right)  >0$ such that for all
$u\in X_{a,b}:$
\begin{align*}
&  {\int\limits_{\mathbb{R}^{N}}}\frac{\left\vert \nabla u\right\vert ^{2}%
}{\left\vert x\right\vert ^{2b}}\mathrm{dx}+{\int\limits_{\mathbb{R}^{N}}%
}\frac{\left\vert u\right\vert ^{2}}{\left\vert x\right\vert ^{2a}}%
\mathrm{dx}-\left(  N-a-b-1\right)  {\int\limits_{\mathbb{R}^{N}}}%
\frac{\left\vert u\right\vert ^{2}}{\left\vert x\right\vert ^{a+b+1}%
}\mathrm{dx}\\
&  \geq C_{1}\left(  N,a,b\right)  \inf_{c}{\int\limits_{\mathbb{R}^{N}}}%
\frac{\left\vert u-ce^{-\frac{1}{b+1-a}\left\vert x\right\vert ^{b+1-a}%
}\right\vert ^{2}}{\left\vert x\right\vert ^{\frac{2bN}{N-2}}}\mathrm{dx}.
\end{align*}
Furthermore, if $a+b+1=\frac{2bN}{N-2}$, then there exists a universal
constant $C_{2}\left(  N,a,b\right)  >0$ such that for all $u\in X_{a,b}:$%
\[
\delta_{1,a,b}\left(  u\right)  \geq C_{2}\left(  N,a,b\right)  \inf_{v\in
E_{a,b}}\left(  {\int\limits_{\mathbb{R}^{N}}}\frac{\left\vert u-v\right\vert
^{2}}{\left\vert x\right\vert ^{a+b+1}}\mathrm{dx}\right)  .
\]

\end{theorem}

\begin{proof}
By applying Lemma \ref{key2} with $\mu=2b,$ $\alpha=b+1-a$ and $\delta
=\frac{2}{b+1-a}$, we have%
\begin{align*}
&  {\int\limits_{\mathbb{R}^{N}}}\frac{\left\vert \nabla u\right\vert ^{2}%
}{\left\vert x\right\vert ^{2b}}\mathrm{dx}+{\int\limits_{\mathbb{R}^{N}}%
}\frac{\left\vert u\right\vert ^{2}}{\left\vert x\right\vert ^{2a}}%
\mathrm{dx}-\left(  N-a-b-1\right)  {\int\limits_{\mathbb{R}^{N}}}%
\frac{\left\vert u\right\vert ^{2}}{\left\vert x\right\vert ^{a+b+1}%
}\mathrm{dx}\\
&  ={\int\limits_{\mathbb{R}^{N}}}\frac{1}{\left\vert x\right\vert ^{2b}%
}\left\vert \nabla\left(  ue^{\frac{1}{b+1-a}\left\vert x\right\vert ^{b+1-a}%
}\right)  \right\vert ^{2}e^{-\frac{2}{b+1-a}\left\vert x\right\vert ^{b+1-a}%
}\mathrm{dx}\\
&  \geq C_{1}\left(  N,a,b\right)  \inf_{c}{\int\limits_{\mathbb{R}^{N}}}%
\frac{\left\vert ue^{\frac{1}{b+1-a}\left\vert x\right\vert ^{b+1-a}%
}-c\right\vert ^{2}}{\left\vert x\right\vert ^{\frac{2bN}{N-2}}}e^{-\frac
{2}{b+1-a}\left\vert x\right\vert ^{b+1-a}}\mathrm{dx}\\
&  =C_{1}\left(  N,a,b\right)  \inf_{c}{\int\limits_{\mathbb{R}^{N}}}%
\frac{\left\vert u-ce^{-\frac{1}{b+1-a}\left\vert x\right\vert ^{b+1-a}%
}\right\vert ^{2}}{\left\vert x\right\vert ^{\frac{2bN}{N-2}}}\mathrm{dx}.
\end{align*}
Now, if $a+b+1=\frac{2bN}{N-2}$, then by applying the above result for
$u_{\lambda}\left(  x\right)  =u\left(  \lambda x\right)  $ and noting that%
\begin{align*}
{\int\limits_{\mathbb{R}^{N}}}\frac{\left\vert \nabla u_{\lambda}\right\vert
^{2}}{\left\vert x\right\vert ^{2b}}\mathrm{dx}  &  =\lambda^{2+2b-N}%
{\int\limits_{\mathbb{R}^{N}}}\frac{\left\vert \nabla u\right\vert ^{2}%
}{\left\vert x\right\vert ^{2b}}\mathrm{dx}\\
{\int\limits_{\mathbb{R}^{N}}}\frac{\left\vert u_{\lambda}\right\vert ^{2}%
}{\left\vert x\right\vert ^{2a}}\mathrm{dx}  &  =\lambda^{2a-N}{\int
\limits_{\mathbb{R}^{N}}}\frac{\left\vert u\right\vert ^{2}}{\left\vert
x\right\vert ^{2a}}\mathrm{dx}\\
{\int\limits_{\mathbb{R}^{N}}}\frac{\left\vert u_{\lambda}\right\vert ^{2}%
}{\left\vert x\right\vert ^{a+b+1}}\mathrm{dx}  &  =\lambda^{a+b+1-N}%
{\int\limits_{\mathbb{R}^{N}}}\frac{\left\vert u\right\vert ^{2}}{\left\vert
x\right\vert ^{2a}}\mathrm{dx}\\
{\int\limits_{\mathbb{R}^{N}}}\frac{\left\vert u_{\lambda}-ce^{-\frac
{1}{b+1-a}\left\vert x\right\vert ^{b+1-a}}\right\vert ^{2}}{\left\vert
x\right\vert ^{\frac{2bN}{N-2}}}\mathrm{dx}  &  =\lambda^{a+b+1-N}%
{\int\limits_{\mathbb{R}^{N}}}\frac{\left\vert u-ce^{-\frac{1}{\left(
b+1-a\right)  \lambda^{b+1-a}}\left\vert x\right\vert ^{b+1-a}}\right\vert
^{2}}{\left\vert x\right\vert ^{\frac{2bN}{N-2}}}\mathrm{dx},
\end{align*}
we get%
\begin{align*}
&  \lambda^{b+1-a}{\int\limits_{\mathbb{R}^{N}}}\frac{\left\vert \nabla
u\right\vert ^{2}}{\left\vert x\right\vert ^{2b}}\mathrm{dx}+\frac{1}%
{\lambda^{b+1-a}}{\int\limits_{\mathbb{R}^{N}}}\frac{\left\vert u\right\vert
^{2}}{\left\vert x\right\vert ^{2a}}\mathrm{dx}-\left(  N-a-b-1\right)
{\int\limits_{\mathbb{R}^{N}}}\frac{\left\vert u\right\vert ^{2}}{\left\vert
x\right\vert ^{a+b+1}}\mathrm{dx}\\
&  \geq C_{1}\left(  N,a,b\right)  \inf_{c}{\int\limits_{\mathbb{R}^{N}}}%
\frac{\left\vert u-ce^{-\frac{1}{\left(  b+1-a\right)  \lambda^{b+1-a}%
}\left\vert x\right\vert ^{b+1-a}}\right\vert ^{2}}{\left\vert x\right\vert
^{\frac{2bN}{N-2}}}\mathrm{dx}.
\end{align*}
By choosing
\[
\lambda=\left(  \frac{{\int\limits_{\mathbb{R}^{N}}}\frac{\left\vert
u\right\vert ^{2}}{\left\vert x\right\vert ^{2a}}\mathrm{dx}}{{\int
\limits_{\mathbb{R}^{N}}}\frac{\left\vert \nabla u\right\vert ^{2}}{\left\vert
x\right\vert ^{2b}}\mathrm{dx}}\right)  ^{\frac{1}{2\left(  b+1-a\right)  }},
\]
we obtain%
\begin{align*}
&  \left(  {\int\limits_{\mathbb{R}^{N}}}\frac{|u|^{2}}{|x|^{2a}}%
\mathrm{dx}\right)  ^{\frac{1}{2}}\left(  {\int\limits_{\mathbb{R}^{N}}}%
\frac{|\nabla u|^{2}}{|x|^{2b}}\mathrm{dx}\right)  ^{\frac{1}{2}}-\left\vert
\frac{N-a-b-1}{2}\right\vert \left(  {\int\limits_{\mathbb{R}^{N}}}%
\frac{|u|^{2}}{|x|^{a+b+1}}\mathrm{dx}\right) \\
&  \geq C_{2}\left(  N,a,b\right)  \inf_{v\in E_{a,b}}\left(  {\int
\limits_{\mathbb{R}^{N}}}\frac{\left\vert u-v\right\vert ^{2}}{\left\vert
x\right\vert ^{a+b+1}}\mathrm{dx}\right)  .
\end{align*}

\end{proof}

Obviously, if we use the Caffarelli-Kohn-Nirenberg deficit%
\[
\delta_{2,a,b}\left(  u\right)  :=\left(  {\int\limits_{\mathbb{R}^{N}}}%
\frac{|u|^{2}}{|x|^{2a}}\mathrm{dx}\right)  \left(  {\int\limits_{\mathbb{R}%
^{N}}}\frac{|\nabla u|^{2}}{|x|^{2b}}\mathrm{dx}\right)  -\left\vert
\frac{N-a-b-1}{2}\right\vert ^{2}\left(  {\int\limits_{\mathbb{R}^{N}}}%
\frac{|u|^{2}}{|x|^{a+b+1}}\mathrm{dx}\right)  ^{2},
\]
then we get

\begin{theorem}
Let $0\leq b<\frac{N-2}{2}$, $a\leq\frac{Nb}{N-2}$ and $a+b+1=\frac{2bN}{N-2}%
$. There exists a universal constant $C_{3}\left(  N,a,b\right)  >0$ such that
for all $u\in X_{a,b}:$%
\begin{align*}
\delta_{2,a,b}\left(  u\right)   &  \geq C_{3}\left(  N,a,b\right)  \left(
{\int\limits_{\mathbb{R}^{N}}}\frac{|u|^{2}}{|x|^{a+b+1}}\mathrm{dx}\right)
d_{1,a,b}\left(  u,E_{a,b}\right)  ^{2}\\
&  +C_{3}\left(  N,a,b\right)  d_{1,a,b}\left(  u,E_{a,b}\right)  ^{4}\text{.}%
\end{align*}

\end{theorem}

\begin{proof}
Since
\[
\delta_{1,a,b}\left(  u\right)  \geq C_{2}\left(  N,a,b\right)  d_{a,b}%
^{2}\left(  u,E_{a,b}\right)  ,
\]
we get%
\begin{align*}
\left(  {\int\limits_{\mathbb{R}^{N}}}\frac{|u|^{2}}{|x|^{2a}}\mathrm{dx}%
\right)  ^{\frac{1}{2}}\left(  {\int\limits_{\mathbb{R}^{N}}}\frac{|\nabla
u|^{2}}{|x|^{2b}}\mathrm{dx}\right)  ^{\frac{1}{2}}  &  \geq\left\vert
\frac{N-a-b-1}{2}\right\vert \left(  {\int\limits_{\mathbb{R}^{N}}}%
\frac{|u|^{2}}{|x|^{a+b+1}}\mathrm{dx}\right) \\
&  +C_{2}\left(  N,a,b\right)  d_{1,a,b}\left(  u,E_{a,b}\right)  ^{2}.
\end{align*}
Hence%
\begin{align*}
&  \left(  {\int\limits_{\mathbb{R}^{N}}}\frac{|u|^{2}}{|x|^{2a}}%
\mathrm{dx}\right)  \left(  {\int\limits_{\mathbb{R}^{N}}}\frac{|\nabla
u|^{2}}{|x|^{2b}}\mathrm{dx}\right)  -\left\vert \frac{N-a-b-1}{2}\right\vert
^{2}\left(  {\int\limits_{\mathbb{R}^{N}}}\frac{|u|^{2}}{|x|^{a+b+1}%
}\mathrm{dx}\right)  ^{2}\\
&  \geq2\left\vert \frac{N-a-b-1}{2}\right\vert C_{2}\left(  N,a,b\right)
\left(  {\int\limits_{\mathbb{R}^{N}}}\frac{|u|^{2}}{|x|^{a+b+1}}%
\mathrm{dx}\right)  d_{1,a,b}\left(  u,E_{a,b}\right)  ^{2}\\
&  +C_{2}^{2}\left(  N,a,b\right)  d_{1,a,b}\left(  u,E_{a,b}\right)  ^{4}.
\end{align*}

\end{proof}

A more careful analysis now leads to the following quantitative version of the
stability of the Caffarelli-Kohn-Nirenberg inequalities:

\begin{theorem}
\label{T3.7}Let $0\leq b<\frac{N-2}{2}$, $a\leq\frac{Nb}{N-2}$ and
$a+b+1=\frac{2bN}{N-2}$. There exists a universal constant $C_{4}\left(
N,a,b\right)  >0$ such that for all $u\in X_{a,b}:$
\[
\delta_{1,a,b}\left(  u\right)  \geq C_{4}\left(  N,a,b\right)  \inf_{u^{\ast
}\in E_{a,b}}\left\{  {\int\limits_{\mathbb{R}^{N}}}\frac{\left\vert
u-u^{\ast}\right\vert ^{2}}{\left\vert x\right\vert ^{a+b+1}}\mathrm{dx}%
:{\int\limits_{\mathbb{R}^{N}}}\frac{|u|^{2}}{|x|^{a+b+1}}\mathrm{dx}%
={\int\limits_{\mathbb{R}^{N}}}\frac{|u^{\ast}|^{2}}{|x|^{a+b+1}}%
\mathrm{dx}\right\}
\]
and
\begin{align*}
\delta_{2,a,b}\left(  u\right)   &  \geq C_{4}\left(  N,a,b\right)  \left(
{\int\limits_{\mathbb{R}^{N}}}\frac{|u|^{2}}{|x|^{a+b+1}}\mathrm{dx}\right)
d_{2,a,b}\left(  u,E_{a,b}\right)  ^{2}\\
&  +C_{4}\left(  N,a,b\right)  d_{2,a,b}\left(  u,E_{a,b}\right)  ^{4}\text{.}%
\end{align*}
Here $d_{2,a,b}\left(  u,E_{a,b}\right)  :=\inf_{v\in E_{a,b}}\left\{  \left(
{\int\limits_{\mathbb{R}^{N}}}\frac{\left\vert u-v\right\vert ^{2}}{\left\vert
x\right\vert ^{a+b+1}}\mathrm{dx}\right)  ^{\frac{1}{2}}:{\int
\limits_{\mathbb{R}^{N}}}\frac{|u|^{2}}{|x|^{a+b+1}}\mathrm{dx}={\int
\limits_{\mathbb{R}^{N}}}\frac{|u^{\ast}|^{2}}{|x|^{a+b+1}}\mathrm{dx}%
\right\}  .$
\end{theorem}

\begin{proof}
WLOG, asssume that ${\int\limits_{\mathbb{R}^{N}}}\frac{|u|^{2}}{|x|^{a+b+1}}%
$\textrm{$dx$}$=1$.

Now, recall that
\begin{align*}
&  \left(  {\int\limits_{\mathbb{R}^{N}}}\frac{|u|^{2}}{|x|^{2a}}%
\mathrm{dx}\right)  ^{\frac{1}{2}}\left(  {\int\limits_{\mathbb{R}^{N}}}%
\frac{|\nabla u|^{2}}{|x|^{2b}}\mathrm{dx}\right)  ^{\frac{1}{2}}-\left\vert
\frac{N-a-b-1}{2}\right\vert \left(  {\int\limits_{\mathbb{R}^{N}}}%
\frac{|u|^{2}}{|x|^{a+b+1}}\mathrm{dx}\right) \\
&  \geq C_{2}\left(  N,a,b\right)  \inf_{u^{\ast}\in E_{a,b}}{\int
\limits_{\mathbb{R}^{N}}}\frac{\left\vert u-u^{\ast}\right\vert ^{2}%
}{\left\vert x\right\vert ^{a+b+1}}\mathrm{dx}.
\end{align*}

Now, if
\[
\delta_{1,a,b}\left(  u\right)  =\left(  {\int\limits_{\mathbb{R}^{N}}}%
\frac{|u|^{2}}{|x|^{2a}}\mathrm{dx}\right)  ^{\frac{1}{2}}\left(
{\int\limits_{\mathbb{R}^{N}}}\frac{|\nabla u|^{2}}{|x|^{2b}}\mathrm{dx}%
\right)  ^{\frac{1}{2}}-\left\vert \frac{N-a-b-1}{2}\right\vert \left(
{\int\limits_{\mathbb{R}^{N}}}\frac{|u|^{2}}{|x|^{a+b+1}}\mathrm{dx}\right)
<\frac{C_{2}\left(  N,a,b\right)  }{2},
\]
then there exists $v\in E$, $v\neq0$ such that
\[
{\int\limits_{\mathbb{R}^{N}}}\frac{\left\vert u-v\right\vert ^{2}}{\left\vert
x\right\vert ^{a+b+1}}\mathrm{dx}\leq\frac{2}{C_{2}\left(  N,a,b\right)
}\delta_{1,a,b}\left(  u\right)  <1.
\]
Let $\lambda=\left(  {\int\limits_{\mathbb{R}^{N}}}\frac{|v|^{2}}{|x|^{a+b+1}%
}\mathrm{dx}\right)  ^{-\frac{1}{2}}$. Then $w=\lambda v\in E_{a,b}$ and
${\int\limits_{\mathbb{R}^{N}}}\frac{|w|^{2}}{|x|^{a+b+1}}$\textrm{$dx$}$=1$.
Also,
\begin{align*}
{\int\limits_{\mathbb{R}^{N}}}\frac{\left\vert u-w\right\vert ^{2}%
}{|x|^{a+b+1}}\mathrm{dx}  &  ={\int\limits_{\mathbb{R}^{N}}}\frac{\left\vert
u-\lambda v\right\vert ^{2}}{|x|^{a+b+1}}\mathrm{dx}\\
&  ={\int\limits_{\mathbb{R}^{N}}}\frac{\left\vert u-v+\left(  1-\lambda
\right)  v\right\vert ^{2}}{|x|^{a+b+1}}\mathrm{dx}\\
&  \leq2{\int\limits_{\mathbb{R}^{N}}}\frac{\left\vert u-v\right\vert ^{2}%
}{|x|^{a+b+1}}\mathrm{dx}+2\left(  1-\lambda\right)  ^{2}{\int
\limits_{\mathbb{R}^{N}}}\frac{\left\vert v\right\vert ^{2}}{|x|^{a+b+1}%
}\mathrm{dx}\\
&  \leq\frac{4}{C_{2}\left(  N,a,b\right)  }\delta_{1,a,b}\left(  u\right)
+2\left(  \frac{1}{\lambda}-1\right)  ^{2}\\
&  =\frac{4}{C_{2}\left(  N,a,b\right)  }\delta_{1,a,b}\left(  u\right)
+2\left(  \left(  {\int\limits_{\mathbb{R}^{N}}}\frac{\left\vert v\right\vert
^{2}}{|x|^{a+b+1}}\mathrm{dx}\right)  ^{\frac{1}{2}}-\left(  {\int
\limits_{\mathbb{R}^{N}}}\frac{\left\vert u\right\vert ^{2}}{|x|^{a+b+1}%
}\mathrm{dx}\right)  ^{\frac{1}{2}}\right)  ^{2}\\
&  \leq\frac{4}{C_{2}\left(  N,a,b\right)  }\delta_{1,a,b}\left(  u\right)
+2{\int\limits_{\mathbb{R}^{N}}}\frac{\left\vert u-v\right\vert ^{2}%
}{|x|^{a+b+1}}\mathrm{dx}\\
&  \leq\frac{8}{C_{2}\left(  N,a,b\right)  }\delta_{1,a,b}\left(  u\right)  .
\end{align*}
That is
\[
\delta_{1,a,b}\left(  u\right)  \geq\frac{C_{2}\left(  N,a,b\right)  }{8}%
\inf_{u^{\ast}\in E_{a,b}}\left\{  {\int\limits_{\mathbb{R}^{N}}}%
\frac{\left\vert u-u^{\ast}\right\vert ^{2}}{\left\vert x\right\vert ^{a+b+1}%
}\mathrm{dx}:{\int\limits_{\mathbb{R}^{N}}}\frac{|u|^{2}}{|x|^{a+b+1}%
}\mathrm{dx}={\int\limits_{\mathbb{R}^{N}}}\frac{|u^{\ast}|^{2}}{|x|^{a+b+1}%
}\mathrm{dx}\right\}  .
\]
If
\[
\delta_{1,a,b}\left(  u\right)  =\left(  {\int\limits_{\mathbb{R}^{N}}}%
\frac{|u|^{2}}{|x|^{2a}}\mathrm{dx}\right)  ^{\frac{1}{2}}\left(
{\int\limits_{\mathbb{R}^{N}}}\frac{|\nabla u|^{2}}{|x|^{2b}}\mathrm{dx}%
\right)  ^{\frac{1}{2}}-\left\vert \frac{N-a-b-1}{2}\right\vert \left(
{\int\limits_{\mathbb{R}^{N}}}\frac{|u|^{2}}{|x|^{a+b+1}}\mathrm{dx}\right)
\geq\frac{C_{2}\left(  N,a,b\right)  }{2}%
\]
then since
\[
\inf_{u^{\ast}\in E_{a,b}}\left\{  {\int\limits_{\mathbb{R}^{N}}}%
\frac{\left\vert u-u^{\ast}\right\vert ^{2}}{\left\vert x\right\vert ^{a+b+1}%
}\mathrm{dx}:{\int\limits_{\mathbb{R}^{N}}}\frac{|u|^{2}}{|x|^{a+b+1}%
}\mathrm{dx}={\int\limits_{\mathbb{R}^{N}}}\frac{|u^{\ast}|^{2}}{|x|^{a+b+1}%
}\mathrm{dx}\right\}  \leq4,
\]
we get%
\[
\delta_{1,a,b}\left(  u\right)  \geq\frac{C_{2}\left(  N,a,b\right)  }{8}%
\inf_{u^{\ast}\in E_{a,b}}\left\{  {\int\limits_{\mathbb{R}^{N}}}%
\frac{\left\vert u-u^{\ast}\right\vert ^{2}}{\left\vert x\right\vert ^{a+b+1}%
}\mathrm{dx}:{\int\limits_{\mathbb{R}^{N}}}\frac{|u|^{2}}{|x|^{a+b+1}%
}\mathrm{dx}={\int\limits_{\mathbb{R}^{N}}}\frac{|u^{\ast}|^{2}}{|x|^{a+b+1}%
}\mathrm{dx}\right\}  .
\]

\end{proof}

Our last result is a stronger stability version for the scale non-invariant
Caffarelli-Kohn-Nirenberg inequalities:

\begin{theorem}
\label{T3.8}Let $0\leq b<\frac{N-2}{2}$, $a\leq\frac{Nb}{N-2}$ and
$a+b+1=\frac{2bN}{N-2}$. There exists a universal constant $C_{5}\left(
N,a,b\right)  >0$ such that for all $u\in X_{a,b}:$
\begin{align*}
&  {\int\limits_{\mathbb{R}^{N}}}\frac{\left\vert \nabla u\right\vert ^{2}%
}{\left\vert x\right\vert ^{2b}}\mathrm{dx}+{\int\limits_{\mathbb{R}^{N}}%
}\frac{\left\vert u\right\vert ^{2}}{\left\vert x\right\vert ^{2a}}%
\mathrm{dx}-\left(  N-a-b-1\right)  {\int\limits_{\mathbb{R}^{N}}}%
\frac{\left\vert u\right\vert ^{2}}{\left\vert x\right\vert ^{a+b+1}%
}\mathrm{dx}\\
&  \geq C_{5}\left(  N,a,b\right)  \inf_{c}\left[
\begin{array}
[c]{c}%
{\int\limits_{\mathbb{R}^{N}}}\frac{\left\vert u-ce^{-\frac{1}{b+1-a}%
\left\vert x\right\vert ^{b+1-a}}\right\vert ^{2}}{\left\vert x\right\vert
^{a+b+1}}\mathrm{dx}+{\int\limits_{\mathbb{R}^{N}}}\frac{\left\vert
\nabla\left(  u-ce^{-\frac{1}{b+1-a}\left\vert x\right\vert ^{b+1-a}}\right)
\right\vert ^{2}}{\left\vert x\right\vert ^{2b}}\mathrm{dx}\\
+{\int\limits_{\mathbb{R}^{N}}}\frac{\left\vert u-ce^{-\frac{1}{b+1-a}%
\left\vert x\right\vert ^{b+1-a}}\right\vert ^{2}}{\left\vert x\right\vert
^{2a}}\mathrm{dx}%
\end{array}
\right]
\end{align*}

\end{theorem}

\begin{proof}
Indeed, let $u=ve^{-\frac{1}{b+1-a}\left\vert x\right\vert ^{b+1-a}}$, then
\begin{align*}
&  {\int\limits_{\mathbb{R}^{N}}}\frac{\left\vert \nabla\left(  u-ce^{-\frac
{1}{b+1-a}\left\vert x\right\vert ^{b+1-a}}\right)  \right\vert ^{2}%
}{\left\vert x\right\vert ^{2b}}\mathrm{dx}\\
&  ={\int\limits_{\mathbb{R}^{N}}}\frac{\left\vert \nabla\left[  \left(
v-c\right)  e^{-\frac{1}{b+1-a}\left\vert x\right\vert ^{b+1-a}}\right]
\right\vert ^{2}}{\left\vert x\right\vert ^{2b}}\mathrm{dx}\\
&  ={\int\limits_{\mathbb{R}^{N}}}\frac{\left\vert \nabla v\right\vert
^{2}e^{-\frac{2}{b+1-a}\left\vert x\right\vert ^{b+1-a}}}{\left\vert
x\right\vert ^{2b}}\mathrm{dx}-2{\int\limits_{\mathbb{R}^{N}}}\frac{\left(
v-c\right)  x\cdot\nabla\left(  v-c\right)  e^{-\frac{2}{b+1-a}\left\vert
x\right\vert ^{b+1-a}}}{\left\vert x\right\vert ^{a+b+1}}\mathrm{dx}\\
&  +{\int\limits_{\mathbb{R}^{N}}}\frac{\left\vert v-c\right\vert
^{2}e^{-\frac{2}{b+1-a}\left\vert x\right\vert ^{b+1-a}}}{\left\vert
x\right\vert ^{2a}}\mathrm{dx}\\
&  ={\int\limits_{\mathbb{R}^{N}}}\frac{\left\vert \nabla v\right\vert
^{2}e^{-\frac{2}{b+1-a}\left\vert x\right\vert ^{b+1-a}}}{\left\vert
x\right\vert ^{2b}}\mathrm{dx}+\left(  N-a-b-1\right)  {\int
\limits_{\mathbb{R}^{N}}}\frac{\left\vert v-c\right\vert ^{2}}{\left\vert
x\right\vert ^{a+b+1}}e^{-\frac{2}{b+1-a}\left\vert x\right\vert ^{b+1-a}%
}\mathrm{dx}\\
&  -{\int\limits_{\mathbb{R}^{N}}}\frac{\left\vert v-c\right\vert
^{2}e^{-\frac{2}{b+1-a}\left\vert x\right\vert ^{b+1-a}}}{\left\vert
x\right\vert ^{2a}}\mathrm{dx}.
\end{align*}
Hence%
\begin{align*}
&  {\int\limits_{\mathbb{R}^{N}}}\frac{\left\vert u-ce^{-\frac{1}%
{b+1-a}\left\vert x\right\vert ^{b+1-a}}\right\vert ^{2}}{\left\vert
x\right\vert ^{a+b+1}}\mathrm{dx}+{\int\limits_{\mathbb{R}^{N}}}%
\frac{\left\vert \nabla\left(  u-ce^{-\frac{1}{b+1-a}\left\vert x\right\vert
^{b+1-a}}\right)  \right\vert ^{2}}{\left\vert x\right\vert ^{2b}}%
\mathrm{dx}+{\int\limits_{\mathbb{R}^{N}}}\frac{\left\vert u-ce^{-\frac
{1}{b+1-a}\left\vert x\right\vert ^{b+1-a}}\right\vert ^{2}}{\left\vert
x\right\vert ^{2a}}\mathrm{dx}\\
&  ={\int\limits_{\mathbb{R}^{N}}}\frac{\left\vert \nabla\left[  \left(
v-c\right)  e^{-\frac{1}{b+1-a}\left\vert x\right\vert ^{b+1-a}}\right]
\right\vert ^{2}}{\left\vert x\right\vert ^{2b}}\mathrm{dx}+{\int
\limits_{\mathbb{R}^{N}}}\frac{\left\vert v-c\right\vert ^{2}}{\left\vert
x\right\vert ^{a+b+1}}e^{-\frac{2}{b+1-a}\left\vert x\right\vert ^{b+1-a}%
}\mathrm{dx}+{\int\limits_{\mathbb{R}^{N}}}\frac{\left\vert v-c\right\vert
^{2}e^{-\frac{2}{b+1-a}\left\vert x\right\vert ^{b+1-a}}}{\left\vert
x\right\vert ^{2a}}\mathrm{dx}\\
&  \leq{\int\limits_{\mathbb{R}^{N}}}\frac{\left\vert \nabla v\right\vert
^{2}e^{-\frac{2}{b+1-a}\left\vert x\right\vert ^{b+1-a}}}{\left\vert
x\right\vert ^{2b}}\mathrm{dx}+\left(  N-a-b\right)  {\int\limits_{\mathbb{R}%
^{N}}}\frac{\left\vert v-c\right\vert ^{2}}{\left\vert x\right\vert ^{a+b+1}%
}e^{-\frac{2}{b+1-a}\left\vert x\right\vert ^{b+1-a}}\mathrm{dx}%
\end{align*}
Therefore, by the weighted Poincar\'{e} inequality Lemma \ref{key2}, we
obtain
\begin{align*}
&  \inf_{c\in\mathbb{R}}{\int\limits_{\mathbb{R}^{N}}}\frac{\left\vert
u-ce^{-\frac{1}{b+1-a}\left\vert x\right\vert ^{b+1-a}}\right\vert ^{2}%
}{\left\vert x\right\vert ^{a+b+1}}\mathrm{dx}+{\int\limits_{\mathbb{R}^{N}}%
}\frac{\left\vert \nabla\left(  u-ce^{-\frac{1}{b+1-a}\left\vert x\right\vert
^{b+1-a}}\right)  \right\vert ^{2}}{\left\vert x\right\vert ^{2b}}%
\mathrm{dx}+{\int\limits_{\mathbb{R}^{N}}}\frac{\left\vert u-ce^{-\frac
{1}{b+1-a}\left\vert x\right\vert ^{b+1-a}}\right\vert ^{2}}{\left\vert
x\right\vert ^{2a}}\mathrm{dx}\\
&  \leq{\int\limits_{\mathbb{R}^{N}}}\frac{\left\vert \nabla v\right\vert
^{2}e^{-\frac{2}{b+1-a}\left\vert x\right\vert ^{b+1-a}}}{\left\vert
x\right\vert ^{2b}}\mathrm{dx}+\left(  N-a-b\right)  \inf_{c\in\mathbb{R}%
}{\int\limits_{\mathbb{R}^{N}}}\frac{\left\vert v-c\right\vert ^{2}%
}{\left\vert x\right\vert ^{a+b+1}}e^{-\frac{2}{b+1-a}\left\vert x\right\vert
^{b+1-a}}\mathrm{dx}\\
&  \leq C\left(  N,a,b\right)  {\int\limits_{\mathbb{R}^{N}}}\frac{\left\vert
\nabla v\right\vert ^{2}e^{-\frac{2}{b+1-a}\left\vert x\right\vert ^{b+1-a}}%
}{\left\vert x\right\vert ^{2b}}\mathrm{dx}\\
&  =C\left(  N,a,b\right)  {\int\limits_{\mathbb{R}^{N}}}\frac{1}{\left\vert
x\right\vert ^{2b}}\left\vert \nabla\left(  ue^{\frac{1}{b+1-a}\left\vert
x\right\vert ^{b+1-a}}\right)  \right\vert ^{2}e^{-\frac{2}{b+1-a}\left\vert
x\right\vert ^{b+1-a}}\mathrm{dx}\\
&  =C\left(  N,a,b\right)  \left[  {\int\limits_{\mathbb{R}^{N}}}%
\frac{\left\vert \nabla u\right\vert ^{2}}{\left\vert x\right\vert ^{2b}%
}+{\int\limits_{\mathbb{R}^{N}}}\frac{\left\vert u\right\vert ^{2}}{\left\vert
x\right\vert ^{2a}}\mathrm{dx}-\left(  N-a-b-1\right)  {\int
\limits_{\mathbb{R}^{N}}}\frac{\left\vert u\right\vert ^{2}}{\left\vert
x\right\vert ^{a+b+1}}\right]  .
\end{align*}

\end{proof}

\end{document}